\documentclass[11pt,a4paper]{article}
\usepackage{amssymb}
\usepackage{amsmath}
\usepackage{latexsym}
\usepackage{graphicx} 
\usepackage{hyperref}
\usepackage{amsthm}
\usepackage{verbatim}
\usepackage{psfrag}
\usepackage{bbm}
\usepackage{dsfont}
\usepackage{mathrsfs}
\usepackage{tikz,pgfplots}
\usepackage{yfonts}
\usepackage{color}
\usepackage{enumerate}
\usepackage{subcaption}
\usepackage{caption}
\usepackage{pgfplots}
\usepackage{subcaption}
\usepackage{multirow}
\usepackage[ruled,vlined]{algorithm}
\newcommand*{\collauthor}[2]{{#1}$^{#2}$}

\newcommand*{\affiliation}[2]{$\mbox{}^{{#2}}${#1}}
\newcommand*{\colltitle}[1]{\textbf{#1}}

\newtheorem{pro}{Proposition}[section]
\newtheorem{theorem}[pro]{Theorem}

\newenvironment{keywords}[1]{\vspace{1cm}\\{\bf \slshape{Keywords}}\quad\slshape{#1}}{}

\newcommand{\xf}{\tilde{f}}

\newcommand{\norm}[1]{\left\lVert#1\right\rVert}
    
\begin{document}
\begin{center}
\begin{Large}
  \colltitle{Gradient boosting-based numerical methods for high-dimensional backward stochastic differential equations}
\end{Large} 
\vspace*{1.5ex}

\begin{sc}
\begin{large}
\collauthor{Long Teng}{},
\end{large}
\end{sc}
\vspace{1.5ex}

\affiliation{Lehrstuhl f\"ur Angewandte Mathematik und Numerische Analysis,\\
Fakult\"at f\"ur Mathematik und Naturwissenschaften,\\
Bergische Universit\"at Wuppertal, Gau{\ss}str. 20, 
42119 Wuppertal, Germany,\linebreak teng@math.uni-wuppertal.de}{} 
\end{center}
\section*{Abstract}
In this work we propose a new algorithm for solving high-dimensional backward stochastic differential equations (BSDEs). Based on the general theta-discretization for the time-integrands, we show how to efficiently use eXtreme Gradient Boosting (XGBoost) regression to approximate the resulting conditional expectations in a quite high dimension. Numerical results illustrate the efficiency and accuracy of our proposed algorithms for solving very high-dimensional (up to $10000$ dimensions) nonlinear BSDEs. 
\begin{keywords}
backward stochastic differential equations (BSDEs), XGBoost, high-dimensional problem, regression
\end{keywords}
\\ \\
{\bf MSC classes:} 65M75, 60H35, 65C30
\section{Introduction}
It is well-known that the curse of dimensionality makes computation of partial differential equations (PDEs) and backward stochastic differential equations (BSDEs) challenging. In this paper we consider BSDEs of the form 
\begin{equation}\label{eq:decoupledbsde}
	\left\{
	\begin{array}{l}
		\,\,\, dX_t = a(t, X_t)\,dt + b(t, X_t)\,dW_t,\quad X_0=x_0,\\
		-dY_t = f(t, X_t, Y_t, Z_t)\,dt - Z_t\,dW_t,\\
		\quad Y_T=\xi=g(X_T),
	\end{array}\right.
\end{equation}
where $a:[0, T]\times \mathbb{R}^d \to \mathbb{R}^d,$ $b:[0, T]\times \mathbb{R}^d \to \mathbb{R}^{d\times d}$, $f(t, X_t, Y_t, Z_t): [0, T] \times \mathbb{R}^d \times \mathbb{R} \times \mathbb{R}^{d} \to \mathbb{R}$
is the driver function and $\xi$ is the square-integrable terminal condition.
Many problems (e.g., pricing, hedging) in the field of finance and physics can be represented by such BSDEs, which makes problems easier to solve but exhibits usually no analytical solution, see e.g., \cite{Karoui1997a}. Furthermore, the dimension $d$ can be very high in applications, e.g., $d$ is the number of underlying assets in financial applications. In the case of that $f$ is linear, the solutions of high-dimensional problems can efficiently be approximated by the Monte-Carlo based approaches with the aid of the Feynman-Kac formula. Equation \eqref{eq:decoupledbsde} becomes more challenging when $f$ is nonlinear, and $d$ is quite large (several hundreds), the classical approaches, such as finite difference methods, finite element methods, and nested Monte-Carlo methods suffer from the curse of dimensionality, i.e., their complexity grows exponentially in the dimension.

Recently, several approximation algorithms have been proposed to solve high-dimensional $(\geq 100~dim)$ nonlinear BSDEs. The fully history recursive multilevel Picard approximation (MLP) method has been proposed in \cite{E2019}, and this method has been further studied in e.g., \cite{Becker2020, Hutzenthaler2020a, Hutzenthaler2020b} for solving high-dimensional PDEs. Another class of approximation algorithms are the deep learning-based approximation methods, for which we refer to \cite{Beck2019a, Beck2019b, E2017, Han2017, Ji2020, Kapllani2020}. 
Note that, for both the classes (the MLP and deep learning based method) above we only mention the references, in which the high dimensional nonlinear problems ($\geq 100$ dim) are dealt with and shown. There are also many other attempts in those classes in the literature to solve high-dimensional (up to $50$ dim) BSDEs, for this we refer \cite{Beck2020, Germain2021} for a nice overview. 

The approximation algorithms in the references mentioned above are based on a reformulation, e.g., PDE as a suitable stochastic fixed point equation or stochastic control problem, and then with a forward discretization of the BSDE. For backward deep learning based approximation algorithms we refer to \cite{Germain2020, Pham2021, Hure2020}, in which $15, 20$ and $50$ dimensional numerical examples are considered, respectively. In \cite{Teng2019}, based on the backward theta-discretization for the time-integrands, the resulting conditional expectations are approximated using the regression tree, in which several $100$ dimensional numerical experiments are shown. To approximate those conditional expecations on spatial discretization we refer to \cite{Ruijter2015} for the Fourier method, and e.g., \cite{Teng2020a, Teng2020b, Zhao2014} for the Gaussian quadrature rules.

In this paper, we propose gradient boosting-based backward approximation algorithms for solving high-dimensional nonlinear BSDEs. As in \cite{Teng2019}, we use the general theta-discretization method for the time-integrands and approximate the resulting conditional expectations using the eXtreme Gradient Boosting (XGBoost) regression \cite{Chen2016}. Several numerical experiments of different types of high-dimensional problems are performed to demonstrate the efficiency and accuracy of our proposed algorithms.

In the next section, we start with notation and definitions and discuss in Section \ref{sec:theta} the discretization of time-integrands using the theta-method, and derive the reference equations. Section \ref{sec:xgboost} is devoted to how to use the XGBoost regression to approximate the conditional expectations.
In Section \ref{sec:numexp}, several numerical experiments on different types of quite high-dimensional BSDEs including financial applications are provided.
Finally, Section \ref{sec:conclusion} concludes this work.

\section{Preliminaries}
Throughout the paper, we assume that $(\Omega, \mathcal{F}, P; \{\mathcal{F}_t\}_{0\leq t\leq T})$ is a complete, filtered probability space.
In this space, a standard $d$-dimensional Brownian motion $W_t$ with a finite terminal time $T$ is defined, which generates the filtration $\{\mathcal{F}_t\}_{0\leq t\leq T},$ i.e., 
 $\mathcal{F}_t=\sigma\{W_s, 0\leq s \leq t\}$ for BSDEs. And the usual hypotheses should be satisfied. We denote the set of all $\mathcal{F}_t$-adapted and square
integrable processes in $\mathbb{R}^d$ with $L^2=L^2(0,T;\mathbb{R}^d).$
A pair of process $(Y_t,Z_t)$ is the solution of the BSDEs \eqref{eq:decoupledbsde}
if it is $\mathcal{F}_t$-adapted and square integrable and satisfies \eqref{eq:decoupledbsde} as
\begin{equation}\label{eq:generalBSDE_01}
 Y_t = \xi + \int_{t}^T f(s, X_s, Y_s, Z_s)\,ds - \int_t^T Z_s\,dW_s,\quad t \in [0, T],
\end{equation}
where $f(t, X_s, Y_s, Z_s): [0, T]~\times \mathbb{R}^d \times \mathbb{R} \times \mathbb{R}^{d} \to \mathbb{R}$ is $\mathcal{F}_t$ adapted,
$\xi=g(X_T): \mathbb{R}^d \to \mathbb{R}.$ These solutions exist uniquely under Lipschitz conditions, see \cite{Pardoux1990, Pardoux1992}.

Suppose that the terminal value $Y_T$ is of the form $g(X^{t,x}_T),$ where $X^{t,x}_T$ denotes the solution of $dX_t$ in \eqref{eq:decoupledbsde} starting
from $x$ at time $t.$
Then the solution $(Y^{t,x}_t, Z^{t,x}_t)$ of BSDEs \eqref{eq:decoupledbsde} can be represented
\cite{Karoui1997b, Ma2005, Pardoux1992, Peng1991} as
\begin{equation*}
 Y^{t,x}_t = u(t, x), \quad Z^{t,x}_t=b^{\top}(t, x) \nabla u(t, x) \quad \forall t \in [0, T),
\end{equation*}
which is solution of the semi-linear parabolic PDE of the form
\begin{equation*}
\frac{\partial u}{\partial t} +  \sum_i^na_i \partial_i u + \frac{1}{2}\sum_{i,j}^n (bb^T)_{i,j} \partial^2_{i,j} u + f(t, x, u, (\nabla u)b)=0
\end{equation*}
with the terminal condition $u(T,x)=g(x).$ 
In turn, suppose $(Y, Z)$ is the solution of BSDEs, $u(t, x)=Y^{t,x}_t$ is a viscosity solution to the PDEs.

\section{Discretization of the BSDE using theta-method}\label{sec:theta}
For simplicity, we discuss the discretization with one-dimensional processes, namely $d=1.$
And the extension to higher dimensions is possible and straightforward.
We introduce the time partition for the time interval $[0, T]$
\begin{equation*}
\Delta_t = \{t_i | t_i \in [0, T], i=0,1,\cdots,N_T, t_i<t_{i+1}, t_0=0, t_{N_T}=T \}.
 \end{equation*}
Let $\Delta t_i = t_{i+1}-t_i$ be the time step, and denote the maximum time step with $\Delta t.$ One needs to additionally discretize the forward SDE in \eqref{eq:decoupledbsde} 
\begin{equation}\label{eq:X_int}
 X_t = x_0 + \int_0^t a(s, X_s)\,ds + \int_0^t b(s, X_s)\,dW_s.
\end{equation}
Suppose that the forward SDE \eqref{eq:X_int} can be already discretized by a process $X^{\Delta_t}_{t_i}$
such that
\begin{equation*}
 E\left[\max_{t_i}\left|X_{t_i}-X^{\Delta_t}_{t_i}\right|^2\right]=\mathcal{O}({\Delta_t})
\end{equation*}
which means strong mean square convergence of order $1/2.$
In the case of that $X_t$ follows a known distribution (e.g., geometric Brownian motion), one can obtain good samples on $\Delta_t$ using the known distribution, otherwise the Euler scheme can be employed.

For the backward process \eqref{eq:generalBSDE_01}, the well-known generalized $\theta$-discretization for $Z$ reads \cite{Zhao2009, Zhao2012}
\begin{equation}\label{eq:disBSDEZ_02}
\begin{split}
  &-E_i[Y_{i+1}\Delta W_{i+1}] = \Delta t_i(1-\theta_1)E_i[f(t_{i+1},\mathbb{X}_{i+1})\Delta W_{i+1}]-\Delta t_i \theta_2 Z_i\\
  &\hspace*{2cm}-\Delta t_i(1-\theta_2)E_i[Z_{i+1}] + R_{\theta}^{Z_i},\\
 \end{split}
\end{equation}
where $\theta_1 \in[0, 1], \theta_2\in(0,1]$ and $R_{\theta}^{Z_i}$ is the discretization error. Therefore, the equation \eqref{eq:disBSDEZ_02}
lead to a time discrete approximation $Z^{\Delta_t}$ for $Z$
\begin{equation*}
Z_i^{\Delta_t}=\frac{\theta^{-1}_2}{\Delta t_i}E_i[Y_{i+1}^{\Delta_t}\Delta W_{i+1}] + \theta^{-1}_2(1-\theta_1)E_i[f(t_{i+1},\mathbb{X}^{\Delta_t}_{i+1})\Delta W_{i+1}]-\theta^{-1}_2(1-\theta_2)E_i[Z_{i+1}^{\Delta_t}]
\end{equation*}
And for $Y$ one has
\begin{equation}\label{eq:yieEiYi1}
  Y_i  = E_i[Y_{i+1}] + \Delta{t_i} \theta_3 f(t_i, \mathbb{X}_i) + \Delta{t_i} (1-\theta_3) E_i[f(t_{i+1}, \mathbb{X}_{i+1})] + R_{\theta}^{Y_i} ,\quad \theta_3\in[0, 1]
\end{equation}
where $R_{\theta}^{Y_i}$ is the corresponding discretization error. Note that, due to $\mathbb{X}^{\Delta_t}_i=(X_i^{\Delta_t}, Y_i^{\Delta_t}, Z_i^{\Delta_t}),$ \eqref{eq:yieEiYi1} is implicit and can be solved by using
iterative methods, e.g., Newton's method or Picard scheme. 

By choosing the different values for $\theta_1$ and $\theta_2,$ one can obtain different schemes.
For example, one receives the Crank-Nicolson scheme by setting $\theta_1=\theta_2=\theta_3=1/2,$ which is second-order accurate.
When $\theta_1=\theta_2=\theta_3=1,$ the scheme is first-order accurate, see \cite{Zhao2006, Zhao2009, Zhao2013}. Using the regression tree-based method in \cite{Teng2019}, the author focus on the scheme of $\theta_1=1/2, \theta_2=1, \theta_3=1/2$ for numerical examples. With the XGBoost regression in this paper, we suggest to use $\theta_1=\theta_2=\theta_3=1/2$ for a higher accuracy if $g$ is continuously differentiable,  i.e., $Z_{N_T}$ is known analytically. $\nabla g$ denotes the gradient of $g.$
\begin{algorithm}[H]
\begin{eqnarray}
Y_{N_T}^{\Delta_t} &=& g(X_{N_T}^{\Delta_t}),\,Z_{N_T}^{\Delta_t}= \nabla g (X_{N_T}^{\Delta_t}),\label{eq:scheme_1_1}\\
~\mbox{\bf For}~i&=&N_T-1,\cdots,0: \nonumber\\
Z_i^{\Delta_t}&=&\frac{2}{\Delta t_i}E_i[Y_{i+1}^{\Delta_t}\Delta W_{i+1}] + E_i[f(t_{i+1},\mathbb{X}^{\Delta_t}_{i+1})\Delta W_{i+1}]-E_i[Z_{i+1}^{\Delta_t}],\label{eq:scheme_1_2}\\
Y_i^{\Delta_t}&=&E_i[Y_{i+1}^{\Delta_t}] + \frac{\Delta{t_i}}{2} f(t_i, \mathbb{X}^{\Delta_t}_i) + 
\frac{\Delta{t_i}}{2} E_i[f(t_{i+1}, \mathbb{X}^{\Delta_t}_{i+1})],\label{eq:scheme_1_3}
\end{eqnarray}	
	\caption*{\bf Semidiscrete Scheme 1}
\end{algorithm}
Otherwise, the scheme $\theta_1=\theta_2=\theta_3=1$ should be used, in which $Z_{N_T}$ is not needed to start the iteration, i.e.,
\begin{algorithm}[H]
\begin{eqnarray}
Y_{N_T}^{\Delta_t} &=& g(X_{N_T}^{\Delta_t}),\label{eq:scheme_2_1}\\
~\mbox{\bf For}~i&=&N_T-1,\cdots,0: \nonumber\\
Z_i^{\Delta_t}&=&\frac{1}{\Delta t_i}E_i[Y_{i+1}^{\Delta_t}\Delta W_{i+1}],\label{eq:scheme_2_2}\\
Y_i^{\Delta_t}&=&E_i[Y_{i+1}^{\Delta_t}] + \Delta{t_i} f(t_i, \mathbb{X}^{\Delta_t}_i).\label{eq:scheme_2_3}
\end{eqnarray}
	\caption*{\bf Semidiscrete Scheme 2}
\end{algorithm}
The error estimates for the schemes above are given in Section \ref{sec:erroranalysis}. $X_t$ can be sampled on the grid $\Delta _t$ based on the available distribution or by applying e.g., the Euler method. 

\section{Computation of conditional expectations with the XGBoost regression}\label{sec:xgboost}
Following the idea proposed in \cite{Teng2019}, in this section we firstly introduce how to use the XGBoost regression to approximate the conditional expectations included in the semi-discrete Scheme 1 and 2, and explain why we choose the XGBoost. We then analyze the time complexity and convergence of the proposed algorithms.

\subsection{Non-parametric regression}
We assume that $(X_{i}^{\Delta_t})_{i=0,\cdots,N_t}$ is Markovian. The conditional expectations included in Scheme 1 and 2 are all of the form $E[Y|X]$ for square integrable random variables $X$ and $Y.$ Therefore, we present the XGBoost regression approach based on the form $E[Y|X]$ throughout this section. Suppose that the model in non-parametric regression reads
\begin{equation}
 Y = \eta(X) + \epsilon,
\end{equation}
where $\epsilon$ has a zero expectation and a constant variance, and is independent with $X.$ Obviously, it can be thus implied that
\begin{equation}\label{eq:conditioneta}
 E[Y|X=x]=\eta(x).
\end{equation}
To approximate the conditional expectations, our goal in regression is to find an estimator of this function, $\hat{\eta}(x).$
By non-parametric regression, we are not assuming a particular form for $\eta.$ Instead of, $\hat{\eta}$ is represented by a XGBoost regressor.
Suppose we have a dataset, $(\hat{x}_\mathcal{M}, \hat{y}_\mathcal{M}),\,\mathcal{M}=1,\cdots,M,$ for $(X, Y).$ We split the data into training and test sets, and fit the model, namely XGBoostregressor on the training data. The regressor can be used to determine (predict) $E[Y|X=x]$ for
an arbitrary $x,$ whose value is not necessarily equal to one of samples $\hat{x}_\mathcal{M}.$ 

As an example, we specify the procedure for \eqref{eq:scheme_1_2}, which can be rewritten as
\begin{equation}\label{eq:example_z}
   Z_i^{\Delta_t}=E\left[\frac{2}{\Delta t_i} Y_{i+1}^{\Delta_t}\Delta W_{i+1} + f(t_{i+1},\mathbb{X}^{\Delta_t}_{i+1})\Delta W_{i+1}-Z^{\Delta_t}_{i+1}|X_{i}^{\Delta_t}\right],\quad i=N_T-1,\cdots,0.
\end{equation}

And there exist deterministic functions $z_i^{\Delta_t}(x)$ such that
\begin{equation}\label{eq:funcapproximation}
Z_i^{\Delta_t}=z_i^{\Delta_t}(X_{i}^{\Delta_t}). 
\end{equation}
Starting from the time $T,$ we fit the regressor $\hat{R}_z$ for the conditional expectation in 
\eqref{eq:example_z} using the dataset $(\hat{x}_{N_T-1, \mathcal{M}}, \frac{2}{\Delta t_{N_T-1}}\hat{y}_{N_T, \mathcal{M}}\Delta\hat{w}_{N_T,\mathcal{M}}+\hat{f}_{N_T, \mathcal{M}} \Delta\hat{w}_{N_T,\mathcal{M}}-\hat{z}_{N_T}).$ Thereby, the function
\begin{equation}\label{eq:zfun}
 z_{N_T-1}^{\Delta_t}(x) = E\left[\frac{2}{\Delta t_{N_T-1}}Y_{N_T}^{\Delta_t}\Delta W_{N_T} + f(t_{N_T},\mathbb{X}^{\Delta_t}_{N_T})\Delta W_{N_T}-Z^{\Delta_t}_{N_T}|X_{N_T-1}^{\Delta_t}=x\right],
\end{equation}
is estimated and presented by the regressor, which can predict the dataset $\hat{z}_{N_T-1, \mathcal{M}}$ of the random variable $Z_{N_T-1}^{\Delta_t}$  based on the dataset $\hat{x}_{N_T-1, \mathcal{M}},$ for $\mathcal{M}=1,\cdots, M.$
Recursively, backward in time, the dataset $\hat{z}_{N_T-1, \mathcal{M}}$ (and also $\hat{y}_{N_T-1, \mathcal{M}}$) will be used to generate the dataset $\hat{z}_{N_T-2, \mathcal{M}}$ of the random variables $Z_{N_T-2}^{\Delta_t}$ at the time $t_{N_T-2}.$
At the initial time $t=0,$ we have a fix initial value $x_0$ for $X,$ i.e., a constant dataset. Using the regressor fitted at time $t_1$
we predict the solution $Z_{0}^{\Delta_t}=z_{0}^{\Delta_t}(x_0).$ 
Following the same procedure to the conditional expectations in \eqref{eq:scheme_1_3}, one obtains implicitly $Y_{0}^{\Delta_t}.$

\subsection{The XGBoost regression}\label{sub:binaryreg}
Recently, three third-party gradient boosting algorithms, including XGBoost, Light Gradient Boosting Machine (LightGBM) \cite{Ke2017} and Categorical Boosting (CatBoost)\cite{Dorogush2018} have been developed for classification or regression predictive modelling problems. These ensemble algorithms have been widely used due to their speed and performance. There are many comparisons among those three algorithms in terms of both speed and accuracy. The common outcome is that XGBoost works generally well, in particular in terms of model accuracy, but slower than other two algorithms, LightGBM has usually a highest speed, and the CatBoost performs well only when one has categorical variables in the data and tune them properly. In principle, all the three algorithms can be used for our purpose, the differences on results are caused by differences in the boosting algorithms. We want to do regression without categorical features, and find XGBoost performs better than LightGBM in terms of model accuracy in our experiment with Example 3 (challenging example). Therefore, in this work, we focus on the XGBoost regression. In the sequel of this section we show how to use the XGBoost algorithm \cite{Chen2016} for approximating the conditional expectations in semidiscrete Scheme 1 and 2 by taking \eqref{eq:zfun} as an example.

We denote the predicted conditional expectation $E[\mathcal{Z}|X=\hat{x}_{i, \mathcal{M}}]$ using the XGBoost model fitted with the dataset of $\mathcal{Z}$ with $E^{\hat{x}_{i, \mathcal{M}}}_{i}[\mathcal{Z}],$ $\mathcal{M}=1,\cdots, M,$ $M$ is the sample size. For \eqref{eq:zfun}, using the dataset (samples) of $X_i$ (which are $\hat{x}_{i, \mathcal{M}}$) and the dataset of $\mathcal{Z}_{i+1}=\frac{2}{\Delta t_{i}}Y_{i+1}^{\Delta_t}\Delta W_{i+1} + f(t_{N_T},\mathbb{X}^{\Delta_t}_{i+1})\Delta W_{i+1}-Z^{\Delta_t}_{i+1}$
(which are $\hat{\mathcal{Z}}_{i+1, \mathcal{M}}=\frac{2}{\Delta t_{i}}\hat{y}_{i+1, \mathcal{M}}\Delta\hat{w}_{i+1,\mathcal{M}}+\hat{f}_{i+1, \mathcal{M}} \Delta\hat{w}_{i+1,\mathcal{M}}-\hat{z}_{i+1}$) we train a XGBoost model.
Then, $\hat{z}_{i, \mathcal{M}}:=E^{\hat{x}_{i, \mathcal{M}}}_{i}[\mathcal{Z}_{i+1}]$ means the predicted value of $E[\mathcal{Z}_{i+1}|X=\hat{x}_{i, \mathcal{M}}]$ with that fitted XGBoost model for approximating $z_i^{\Delta t}(\hat{x}_{i, \mathcal{M}}),$ see \eqref{eq:funcapproximation}. Therefore, for the $i$-th step, we define
\begin{equation}\label{eq:err_z1}
R_{\mbox{xgb}}^{Z_i}=\frac{1}{M}\sum_{\mathcal{M}=1}^M \left(\hat{z}_{i, \mathcal{M}}-z_i^{\Delta_t}(\hat{x}_{i,\mathcal{M}})\right)^2,
\end{equation}
as the approximation error. Note that $\hat{\mathcal{Z}}_{i+1, \mathcal{M}}=z_i^{\Delta_t}(\hat{x}_{i,\mathcal{M}})+\epsilon_{i,\mathcal{M}}^{z},$ $\epsilon^{z}_i$ has zero expectation and a constant variance $Var^{z}_i$. Therefore, \eqref{eq:err_z1} can be reformulated as
\begin{align}
R_{\mbox{xgb}}^{Z_i}&=\frac{1}{M}\sum_{\mathcal{M}=1}^M \left(\epsilon_{i,\mathcal{M}}^{z}-\left(\hat{z}_{i, \mathcal{M}}-\hat{\mathcal{Z}}_{i+1,\mathcal{M}}\right)\right)^2\\
&\leq 2Var^{z}_i+\frac{2}{M}\sum_{\mathcal{M}=1}^M \left(\hat{z}_{i, \mathcal{M}}-\hat{\mathcal{Z}}_{i+1,\mathcal{M}}\right)^2, 
\end{align}
where the second term refers to the XGBoost regression error which will be analyzed in the following.
\paragraph{Regularized learning objective}
For simplicity we shall focus on $d=1,$ the results can be straightforwardly extended to multidimensional cases. And for a simplified notation we omit the index of the time step, e.g., $\hat{x}_{\mathcal{M}}$ instead of $\hat{x}_{i,\mathcal{M}}.$ 
Suppose that a given dataset with $M$ samples
\begin{equation*}
\mathcal{D}=\left\{(\hat{x}_{\mathcal{M}},\hat{\mathcal{Z}}_{\mathcal{M}})|~|\mathcal{D}|=M,~\hat{x}_{\mathcal{M}},~\hat{\mathcal{Z}}_{\mathcal{M}}\in \mathbb{R}\right\}.
\end{equation*}
A tree ensemble model consists of $K$ regression trees can be constructed to predict the output
\begin{equation*}
	\hat{z}_{\mathcal{M}}=\eta(\hat{x}_\mathcal{M})=\sum_{k=1}^{K}\xf_k(\hat{x}_{\mathcal{M}}),\quad \xf_k \in \mathcal{S},
\end{equation*}
where $\mathcal{S}=\{\xf(x)=\omega_{q(x)},\,q:\mathbb{R}\to \hat{T},\, \omega \in \mathbb{R}^{\hat{T}}\}$ is the space of regression trees. $q$ denotes the structure of each tree that maps an example to the corresponding leaf index, i.e., each $\xf_k$ corresponds to an independent tree structure $q$ and leaf weights $\omega.$ $\omega_j$ represents score on the $j$-th leaf, and $\hat{T}$ is the number of leaves. 

To train the model we optimize the mean squared error (MSE) for regression
\begin{equation*}
	L(\hat{z}_{\mathcal{M}}, \hat{\mathcal{Z}}_{\mathcal{M}}) = \frac{1}{M}\sum_{\mathcal{M}=1}^{M}(\hat{z}_{\mathcal{M}}-\hat{\mathcal{Z}}_{\mathcal{M}})^2.
\end{equation*}
For a regularized objective we define the regularization term
\begin{equation*}
\Omega(\xf) = \gamma \hat{T} + \frac{1}{2}\lambda \norm{w}^2=\gamma \hat{T} + \frac{1}{2}\lambda \sum_{j=1}^{\hat{T}}w_j^2,
\end{equation*}
where $\gamma, \lambda$ are positive regularization parameter, and $w_j$ is the score on the $j$-th leaf. The regularization term controls the complexity of the model which avoids overfitting. Therefore, the regularized objective is given by
\begin{equation}\label{eq:objective}
\mathcal{L}(\eta)= \sum_{\mathcal{M}=1}^M 	L(\hat{z}_{\mathcal{M}}, \hat{\mathcal{Z}}_{\mathcal{M}}) + \sum_{k=1}^K\Omega(f_k),
\end{equation}
which needs to be minimized, $L$ serves as a loss function that measures the difference between the prediction and target.
\paragraph{Gradient Tree Boosting}
In XGBoost, the gradient descent is used to minimize \eqref{eq:objective}, according to which we minimize the following objective by adding $\xf_k$ in a iterative algorithm
\begin{equation}\label{eq:objetivet}
\mathcal{L}^{(k)} = \sum_{\mathcal{M}=1}^{M}L(\hat{\mathcal{Z}}_{\mathcal{M}},\hat{z}_{\mathcal{M}}^{k})+\sum_{j=1}^k\Omega(\xf_j)=\sum_{\mathcal{M}=1}^{M}L(\hat{\mathcal{Z}}_{\mathcal{M}},\hat{z}_{\mathcal{M}}^{(k-1)}+\xf_k(\hat{x}_{\mathcal{Z}}))+\Omega(\xf_k),
\end{equation}
where $\hat{z}_{\mathcal{M}}^{k}=\sum_{j=1}^{k}\xf_j(\mathbf{x_{\mathcal{M}}}),$ and $k=1,\cdots, K.$ This is to say that \eqref{eq:objective} is minimized by greedily adding $\xf_k.$ For this, one calculates a second-order approximation of \eqref{eq:objetivet} as
\begin{equation*}
\mathcal{L}^{(k)} \approx \sum_{\mathcal{M}=1}^{M}\left(L(\hat{\mathcal{Z}}_{\mathcal{M}},\hat{z}_{\mathcal{M}}^{(k-1)})+g_{\mathcal{M}}\xf_k(\hat{x}_{\mathcal{M}})+\frac{1}{2}h_{\mathcal{M}}\xf^2_k(\hat{x}_{\mathcal{M}})\right)+\Omega(\xf_k),
\end{equation*}
where $g_{\mathcal{M}}=\partial_{\hat{z}^{(k-1)}}L(\hat{\mathcal{Z}}_{\mathcal{M}},\hat{z}^{(k-1)})$ and $h_{\mathcal{M}}=\partial^2_{\hat{z}^{(k-1)}}L(\hat{\mathcal{Z}}_{\mathcal{M}},\hat{z}^{(k-1)})$ are first and second order gradients, respectively. By removing the constant terms one obtains the objective at $k$-th step
\begin{equation}\label{eq:aobjetivet}
\mathcal{\tilde{L}}^{(k)} = \sum_{\mathcal{M}=1}^{M}\left(g_{\mathcal{M}}\xf_k(\hat{x}_{\mathcal{M}})+\frac{1}{2}h_{\mathcal{M}}\xf^2_k(\hat{x}_{\mathcal{M}})\right)+\Omega(\xf_k),
\end{equation}
which needs to be optimized by finding a $\xf_k.$

Next, we show how can one find the tree $\xf_k$ to optimize the prediction. We firstly define a tree as
\begin{equation*}
\xf_k(\hat{x}) = w_{q(\hat{x})}, \quad w \in \mathbb{R}^{\hat{T}}.
\end{equation*}
And define $I_j=\{\mathcal{M}|q(\hat{x}_{\mathcal{M}})=j\}$ as the instance set of leaf $j,$ which contains the indices of data points mapped to the $j$-th leaf. Then, \eqref{eq:aobjetivet} can be rewritten as 
\begin{align*}
\mathcal{\tilde{L}}^{(k)} &= \sum_{\mathcal{M}=1}^{M}\left(g_{\mathcal{M}}\xf_k(\hat{x}_{\mathcal{M}})+\frac{1}{2}h_{\mathcal{M}}\xf^2_k(\hat{x}_{\mathcal{M}})\right)+\gamma \hat{T} + \frac{1}{2}\lambda\sum_{j=1}^{\hat{T}}w_j^2\\
&=\sum_{j=1}^{\hat{T}}\left(\left(\sum_{\mathcal{M}\in I_j}g_{\mathcal{M}}\right)w_j+\frac{1}{2}\left(\sum_{\mathcal{M}\in I_j}h_{\mathcal{M}}+\lambda\right)w_j^2\right)+\gamma \hat{T}.
\end{align*}
For a fixed $q(\hat{x}),$ one can easily compute the optimal $w_j$ of leaf $j$ as 
\begin{equation*}
	w^*_{j}=-\frac{\sum_{\mathcal{M}\in I_j}g_{\mathcal{M}}}{\sum_{\mathcal{M}\in I_j}h_{\mathcal{M}}+\lambda},
\end{equation*}
and thus the corresponding optimal value of the objective
\begin{equation}\label{eq:xgberror}
\mathcal{\tilde{L}}^{(k)}(q)	=-\frac{1}{2}\sum_{j=1}^{\hat{T}}\frac{(\sum_{\mathcal{M}\in I_j}g_{\mathcal{M}})^2}{\sum_{\mathcal{M}\in I_j}h_{\mathcal{M}}+\lambda}+\gamma\hat{T},
\end{equation}
which can be used as a scoring function to measure the quality of $q.$ Due to the high computational cost, it is not realistic to enumerate all the possible $q.$ A greedy algorithm proposed in \cite{Chen2016} that finds best splitting point recursively until the maximum depth. We denote the instance sets of left and right nodes after the split by $I_L$ and $I_R$, and $I=I_L \cup I_R.$ The following loss reduction after the split
\begin{equation}
\mathcal{L}_{\mbox{gain}}=\frac{1}{2}\left(\frac{(\sum_{\mathcal{M}\in I_L}g_{\mathcal{M}})^2}{\sum_{\mathcal{M}\in I_L}h_{\mathcal{M}}+\lambda}+\frac{(\sum_{\mathcal{M}\in I_R}g_{\mathcal{M}})^2}{\sum_{\mathcal{M}\in I_R}h_{\mathcal{M}}+\lambda}-\frac{(\sum_{\mathcal{M}\in I}g_{\mathcal{M}})^2}{\sum_{\mathcal{M}\in I}h_{\mathcal{M}}+\lambda}\right)-\gamma
\end{equation}
is used for evaluating the split candidates, where $\lambda$ is the regularization parameter. Based on the best splitting point, one can then prune out the nodes with a negative gain.

We have introduced the mathematics behind XGBoost. For all other techniques used in the implementation to further prevent overfitting (shrinkage \cite{Friedman2002}, column feature subsampling \cite{Breiman2001, Friedman2003}), improve the efficiency (approximate exact greedy algorithm, sparsity-aware splitting, column block for parallel learning) we refer to \cite{Chen2016}.

\subsection{The fully discrete schemes}\label{sec:praticalapp}
In this section we introduce the fully discrete schemes, where the conditional expectations approximated by XGBoost regression. As shown above, we do not need regression for each conditional expectation in semidiscrete Scheme 1 and 2.
Due to the linearity of conditional expectation, we perform XGBoost regression for the combination of the conditional expectations in one equation. Based on semidiscrete Scheme 1, by combining conditional expectations and including all errors we give the following fully discrete scheme 1 as:
\begin{algorithm}[H]
\begin{align}
  \hat{y}_{N_T, \mathcal{M}} &= g(\hat{x}_{N_T, \mathcal{M}}),\,\hat{z}_{N_T, \mathcal{M}}=g_x(\hat{x}_{N_T, \mathcal{M}}), \nonumber\\
  ~\mbox{\bf For}~i&=N_T-1,\cdots,0~,~\mathcal{M}=1,\cdots, M: \nonumber\\
  \hat{z}_{i, \mathcal{M}}&=E^{\hat{x}_{i, \mathcal{M}}}_{i}\left[\frac{2}{\Delta t_i}Y_{i+1}\Delta W_{i+1} + f(t_{i+1},\mathbb{X}_{i+1})\Delta W_{i+1}-Z_{i+1}\right], \label{eq:fullschemez}\\
   \hat{y}_{i, \mathcal{M}}&=E^{\hat{x}_{i, \mathcal{M}}}_{i}\left[Y_{i+1} + \frac{\Delta{t_i}}{2} f(t_{i+1}, \mathbb{X}_{i+1})\right] + 
   \frac{\Delta{t_i}}{2} \hat{f}_{i, \mathcal{M}}.\label{eq:fullschemey}
   \end{align}
   \caption*{\bf Fully discrete Scheme 1}
\end{algorithm} 
The error of approximating conditional expectations by using the XGBoost regression in \eqref{eq:fullschemez}, namely $R_{\mbox{xgb}}^{Z_i}$ is given in \eqref{eq:err_z1} and bounded by $2(Var^{z}_i + \hat{\mathcal{L}}_{\mbox{min}}(\hat{q}^{z}_i)),$ where $\hat{\mathcal{L}}_{\mbox{min}}(\hat{q}^{z}_i))$ is defined via \eqref{eq:xgberror}. $\hat{q}^{z}_i$ denotes representation of the tree structure with the smallest error at $i$-th step. Similarly, for the regression error in \eqref{eq:fullschemey} we define
\begin{equation}\label{eq:error_y1}
	R_{\mbox{xgb}}^{Y_i}=\frac{1}{M}\sum_{\mathcal{M}=1}^M \left(\hat{y}_{i, \mathcal{M}}-y_i^{\Delta_t}(\hat{x}_{i,\mathcal{M}})\right)^2 \leq 2(Var^{y}_i + \hat{\mathcal{L}}_{\mbox{min}}(\hat{q}^{y}_i)).
\end{equation}
Analogously, the fully discrete scheme 2 can be given as:
\begin{algorithm}[H]
\begin{align*}
\hat{y}_{N_T, \mathcal{M}} &= g(\hat{x}_{N_T, \mathcal{M}}),\\
~\mbox{\bf For}~i&=N_T-1,\cdots,0~,~\mathcal{M}=1,\cdots, M:\\
\hat{z}_{i, \mathcal{M}}&=E^{\hat{x}_{i, \mathcal{M}}}_{i}\left[\frac{1}{\Delta t_i}Y_{i+1}\Delta W_{i+1} \right],\\
\hat{y}_{i, \mathcal{M}}&=E^{\hat{x}_{i, \mathcal{M}}}_{i}\left[Y_{i+1} \right] + \Delta{t_i} \hat{f}_{i, \mathcal{M}}.
\end{align*}
\caption*{\bf Fully discrete Scheme 2}
\end{algorithm}
 
\subsection{Time complexity analysis}\label{sec:complexityanalysis}
We denote the maximum depth of the tree by $\tilde{d},$ and the total number of trees by $K.$
The time complexity of the XGBoost reads \cite{Chen2016}
\begin{equation}\label{eq:xgbcomplexity}
\mathcal{O}(K\tilde{d}\,|\mbox{samples}|+|\mbox{samples}|\log B),
\end{equation}
where $|\mbox{samples}|$ denotes number of samples in the training data, and $B$ is the maximum number of rows in each block. Note that $\mathcal{O}(|\mbox{samples}|\log B)$ is the one time preprocessing cost. From \eqref{eq:xgbcomplexity} we straightforwardly deduce that the complexity of approximating one conditional expectation in our scheme is $\mathcal{O}(K\tilde{d}M d+M d\log B).$ Therefore, the time complexity of our proposed scheme is given by
\begin{equation*}
\mathcal{O}(K\tilde{d}M d N_T+M d N_T\log B).
\end{equation*} 

\subsection{Error estimates}\label{sec:erroranalysis}
The error analysis when $\theta_1=1/2,\theta_2=1,\theta_3=1/2$ has been done in \cite{Teng2019}, we generalize it for the general theta-scheme, i.e., which includes the proposed Scheme 1 and 2. Suppose that the errors of iterative method can be neglected by choosing the number of Picard iterations sufficiently high, we consider the discretization and regression errors in the first place.
The errors due to the time-discretization are given by 
\begin{align*}
	\epsilon^{Y_i,\theta}:&=Y_i-Y_i^{\Delta_t},\\
	\epsilon^{Z_i,\theta}:&=Z_i-Z_i^{\Delta_t},\\
	\epsilon^{f_i,\theta}:&=f(t_i, \mathbb{X}_{i})-f(t_i, \mathbb{X}_{i}^{\Delta_t}).
\end{align*}
As the deterministic function $z_i^{\Delta_t}$ given in \eqref{eq:funcapproximation}  we define deterministic function $y_i^{\Delta_t}$
\begin{equation*}
	Y_i^{\Delta_t}=y_i^{\Delta_t}(X_{i}^{\Delta_t}). 
\end{equation*}
These functions are approximated by the regression trees, resulting in the approximations $\hat{y}_i^{\Delta_t}, \hat{z}_i^{\Delta_t}$ with
\begin{equation*}
	\hat{Y}_i^{\Delta_t}=\hat{y}_i^{\Delta_t}(X_{i}^{\Delta_t}) ~\mbox{and}~ \hat{Z}_i^{\Delta_t}=\hat{z}_i^{\Delta_t}(X_{i}^{\Delta_t}), 
\end{equation*}
Thus, we denote the global errors by
\begin{align*}
	\epsilon^{Y_i}:&=Y_i-\hat{Y}_i^{\Delta_t},\\
	\epsilon^{Z_i}:&=Z_i-\hat{Z}_i^{\Delta_t},\\
	\epsilon^{f_i}:&=f(t_i, \mathbb{X}_{i})-f(t_i, \hat{\mathbb{X}}_{i}^{\Delta_t}).
\end{align*}

{\it Assumption 1}
Suppose that $X_0$ is $\mathcal{F}_0$-measurable with $E[|X_0|^2]<\infty,$ and that $a$ and $b$ are $L^2$-measurable in $(t,x) \in [0, T]\times \mathbb{R}^d,$ are linear growth bounded and uniformly Lipschitz continuous, i.e., there exist positive constants $K$ and $L$ such that
\begin{align*}
	|a(t,x)|^2 &\leq K(1+|x|^2),\quad |b(t,x)|^2 \leq K(1+|x|^2),\\	
	|a(t,x)-a(t,y)| &\leq L|x-y|, \quad |b(t,x)-b(t,y)| \leq L|x-y|	
\end{align*}
with $x, y \in \mathbb{R}^d.$

Let $C_b^{l,k,k,k}$ be the set of continuously differentiable functions $f:[0,T] \times \mathbb{R}^n \times \mathbb{R}^m \times \mathbb{R}^{m \times d} \to \mathbb{R}^m$ with uniformly bounded partial derivatives $\partial_t^{l_1}f$ for $\frac{1}{2}\leq l_1 \leq l$ and $\partial_x^{k_1}\partial_y^{k_2}\partial_z^{k_3}f$ for $1\leq k_1+k_2+k_3 \leq k,$ $C_b^{l,k,k}$ and $C_b^{l,k}$ can be analogously defined, and $C_b^k$ be the set of functions $g:\mathbb{R}^d \to \mathbb{R}^m$ with uniformly bounded partial derivatives $\partial_x^{k_1} g$ for $1\leq k_1\leq k.$
We give some remarks concerning related results on the one-step scheme:
\begin{itemize}
	\item Under Assumption 4.1, if $f \in C_b^{2,4,4,4},$ $g\in C_b^{4+\alpha} $ for some $\alpha \in (0,1),$ $a$ and $b$ are bounded, and $a, b\in C_b^{2,4},$ the absolute values of the local errors $R_\theta^{Y_i}~\mbox{and}~R_\theta^{Z_i}$ can be bounded by $C(\Delta t_i)^3$
	in Scheme 1 and by $C(\Delta t_i)^2$ in Scheme 2, where 
	$C$ is a constant which can depend on $T, x_0$ and the bounds of $a, b, f, g$ in \eqref{eq:decoupledbsde}, see e.g., \cite{Yang2017, Zhao2009, Zhao2014, Zhao2012, Zhao2013}. 
	\item For notation convenience we might omit the dependency of local and global errors on state of the BSDEs and the discretization errors of $dX_t,$ namely we assume that $X_i=X_{i}^{\Delta_t}.$ And we focus on 1-dimensional case $(d=1),$ the results can be extended to high-dimensional case.
	\item For the implicit schemes we will apply Picard iterations which converges because of the Lipschitz assumptions on the driver, and for any initial guess when $\Delta t_i$ is small enough. In the following analysis, we consider the equidistant time discretization $\Delta t.$
\end{itemize}

For the $Z$-component $(0\leq i \leq N_T-1)$ we have (see \eqref{eq:disBSDEZ_02})
\begin{equation*}
	\epsilon^{Z_i}=E_i^{x_i}[\frac{1}{\Delta t \theta_2}\epsilon^{Y_{i+1}}\Delta W_{i+1} + \frac{1-\theta_1}{\theta_2}\epsilon^{f_{i+1}} \Delta W_{i+1}-\frac{1-\theta_2}{\theta_2}\epsilon^{Z_{i+1}}] + \frac{R_{\theta}^{Z_i}}{\Delta t \theta_2}{ +R_{\mbox{xgb}}^{Z_i}},
\end{equation*}
where the $\epsilon^{f_{i+1}}$ can be bounded using Lipschitz continuity of $f$ by
\begin{equation*}
	E_i^{x_{i}}[|\epsilon^{f_{i+1}}|^2] \leq E_i^{x_{i}}[|L(|\epsilon^{Y_{i+1}}| + |\epsilon^{Z_{i+1}}|)|^2] \leq 2 L^2 E_i^{x_{i}}[|\epsilon^{Y_{i+1}}|^2 + |\epsilon^{Z_{i+1}}|^2]
\end{equation*}
with Lipschitz constant $L.$ And it holds that
\begin{equation*}
	|E_i^{x_{i}}[\epsilon^{Y_{i+1}} \Delta W_{i+1}]|^2=|E_i^{x_{i}}[(\epsilon^{Y_{i+1}}-E_i^{x_{i}}[\epsilon^{Y_{i+1}}])\Delta W_{i+1}]|^2 \leq \Delta t (E_i^{x_{i}}[|\epsilon^{Y_{i+1}}|^2] -|E_i^{x_{i}}[\epsilon^{Y_{i+1}}]|^2)
\end{equation*}
with Cauchy-Schwarz inequality. Consequently, we calculate
\begin{equation}\label{eq:error_z}
		\begin{split}
			\theta_2^2(\Delta t)^2|\epsilon^{Z_i}|^2 &\leq 2 \Delta t (E_i^{x_{i}}[|\epsilon^{Y_{i+1}}|^2]-|E_i^{x_{i}}[\epsilon^{Y_{i+1}}]|^2)
			+16(1-\theta_1)^2 L^2 (\Delta t)^3 E_i^{x_{i}}[|\epsilon^{Y_{i+1}}|^2+|\epsilon^{Z_{i+1}}|^2]\\
			&+8(\theta_2-1)^2(\Delta t)^2E_i^{x_{i}}[|\epsilon^{Z_{i+1}}|^2] + 8|R_{\theta}^{Z_i}|^2 + 8(\Delta t)^2\theta_2^2 |R_{\mbox{xgb}}^{Z_i}|^2, 
		\end{split} 
\end{equation}
where H\"older's inequality is used.

For the $Y$-component in the implicit scheme we have
\begin{equation*}
	\epsilon^{Y_i}=E^{x_i}_{i}[\epsilon^{Y_{i+1}} + (1-\theta_3) \Delta t  \epsilon^{f_{i+1}}] +  \theta_3 \Delta t \epsilon^{f_{i}} + R^{Y_i}_{\theta} + R_{\mbox{xgb}}^{Y_i}.
\end{equation*}
This error can be bounded by
\begin{equation*}
	|\epsilon^{Y_i}|\leq |E^{x_i}_{i}[\epsilon^{Y_{i+1}}]|+ \theta_3\Delta t L(|\epsilon^{Y_{i}}|+|\epsilon^{Z_{i}}|)
	+(1-\theta_3)\Delta t LE^{x_i}_{i}[|\epsilon^{Y_{i+1}}|+|\epsilon^{Z_{i+1}}|]+|R_{\theta}^{Y_i}|+ |R_{\mbox{xgb}}^{Y_i}|.
\end{equation*}
By the inequality $(a+b)^2\leq a^2 + b^2 + \gamma \Delta t a^2 + \frac{1}{\gamma \Delta t} b^2$ we calculate
\begin{equation}\label{eq:error_yy}
		\begin{split}
			|\epsilon^{Y_i}|^2 \leq &(1+\gamma\Delta t)|E^{x_i}_{i}[\epsilon^{Y_{i+1}}]|^2 + 6\theta_3^2(\Delta t L)^2(|\epsilon^{Y_{i}}|^2+|\epsilon^{Z_{i}}|^2)\\&+6(1-\theta_3)^2(\Delta t L)^2(E^{x_i}_{i}[|\epsilon^{Y_{i+1}}|^2]+E^{x_i}_{i}[|\epsilon^{Z_{i+1}}|^2])+6|R_{\theta}^{Y_i}|^2 + 6|R_{\mbox{xgb}}^{Y_i}|^2\\
			&+\frac{1}{\gamma}\left(6\theta_3^2\Delta t L^2(|\epsilon^{Y_{i}}|^2+|\epsilon^{Z_{i}}|^2)
			+6(1-\theta_3)^2\Delta t L^2 (E^{x_i}_{i}[|\epsilon^{Y_{i+1}}|^2]+E^{x_i}_{i}[|\epsilon^{Z_{i+1}}|^2])\right.\\
			&\left. +\frac{6|R_{\theta}^{Y_i}|^2}{\Delta t}+\frac{6|R_{\mbox{xgb}}^{Y_i}|^2}{\Delta t}\right).
		\end{split}
\end{equation}
\begin{theorem}\label{theo:convergence}
Under Assumption 4.1, if $f \in C_b^{2,4,4,4},$ $g\in C_b^{4+\alpha} $ for some $\alpha \in (0,1),$ $a$ and $b$ are bounded, $a, b\in C_b^{2,4},$ and 
		given
		\begin{equation*}
			E_{N_T-1}^{x_{N_T-1}}[|\epsilon^{Z_{N_T}}|^2] \thicksim \mathcal{O}((\Delta t)^2), \quad E_{N_T-1}^{x_{N_T-1}}[|\epsilon^{Y_{N_T}}|^2] \thicksim \mathcal{O}((\Delta t)^{2}),
		\end{equation*}
		It holds then
		\begin{equation}\label{eq:resultstheorem}
			E_0^{x_0}\left[|\epsilon^{Y_{i}}|^2 + \frac{(8\theta_3^2(\theta_2-1)^2+(1-\theta_3)^2\theta_2^2)\Delta t}{2(1-\theta_3)^2+2\theta_3^2}|\epsilon^{Z_{i}}|^2\right]\leq Q (\Delta t)^{2} + \tilde{Q}\sum_{i+1}^{N_T}\left(\frac{N_T (\mbox{Var}^{\mathcal{Y}}_j)^2}{T}+\frac{T(\mbox{Var}^{\mathcal{Z}}_j)^2}{N_T}\right),
		\end{equation}
		 $0 \leq i \leq N_T-1,$ where $Q$ is a constant which only depend on $T,$ $x_0$ and the bounds of $f, g~\mbox{and}~a, b$ in \eqref{eq:decoupledbsde}, $\tilde{Q}$ is a constant depending on $T,$ $x_0$ and $L,$ and $\mbox{Var}^{\mathcal{Y}}_i$ and $\mbox{Var}_i^{\mathcal{Z}}$ are the bounded constants, and $M$ is the number of samples.
\end{theorem}
\begin{proof}
		By combining both \eqref{eq:error_z} and \eqref{eq:error_yy} we straightforwardly obtain
		\begin{align*}
			E_i^{x_i}[|\epsilon^{Y_i}|^2] &+ \frac{\theta_2^2\Delta t}{2}E_i^{x_i}[|\epsilon^{Z_i}|^2]
			\leq (1+\gamma\Delta t)|E^{x_i}_{i}[\epsilon^{Y_{i+1}}]|^2 + 6\theta_3^2(\Delta t L)^2(E_i^{x_i}[|\epsilon^{Y_{i}}|^2]+E_i^{x_i}[|\epsilon^{Z_{i}}|^2])\nonumber\\
			&+6(1-\theta_3)^2(\Delta t L)^2(E_i^{x_i}[|\epsilon^{Y_{i+1}}|^2]+E_i^{x_i}[|\epsilon^{Z_{i+1}}|^2])+6E_i^{x_i}[|R_{\theta}^{Y_i}|^2]+6E_i^{x_i}[|R_{\mbox{xgb}}^{Y_i}|^2]\nonumber\\
			&+ (E_i^{x_{i}}[|\epsilon^{Y_{i+1}}|^2]-|E_i^{x_{i}}[\epsilon^{Y_{i+1}}]|^2)
			+8(1-\theta_1)^2(\Delta t L)^2 (E_i^{x_{i}}[|\epsilon^{Y_{i+1}}|^2]+E_i^{x_{i}}[|\epsilon^{Z_{i+1}}|^2])\nonumber\\
			&+4(\theta_2-1)^2\Delta tE_i^{x_{i}}[|\epsilon^{Z_{i+1}}|^2]+ 4\frac{E_i^{x_i}[|R_{\theta}^{Z_i}|^2]}{\Delta t} +4 \Delta t \theta_2^2 E_i^{x_i}[|R_{\mbox{xgb}}^{Z_i}|^2]\\
			& +\frac{1}{\gamma}\left(6\theta_3^2\Delta t L^2(E_i^{x_i}[|\epsilon^{Y_{i}}|^2]+E_i^{x_i}[|\epsilon^{Z_{i}}|^2])+ 6(1-\theta_3)^2\Delta t L^2(E_i^{x_i}[|\epsilon^{Y_{i+1}}|^2]+E_i^{x_i}[|\epsilon^{Z_{i+1}}|^2])\right.\nonumber \\
			&\left.+\frac{6E_i^{x_i}[|R_{\theta}^{Y_i}|^2]}{\Delta t}+\frac{6E_i^{x_i}[|R_{\mbox{xgb}}^{Y_i}|^2]}{\Delta t}\right)\nonumber
		\end{align*}
		which implies
		\begin{align*}
			&\left(1-6\theta_3^2(\Delta t L)^2-\frac{6\theta_3^2\Delta t L^2}{\gamma}\right)E_i^{x_i}[|\epsilon^{Y_i}|^2] + 
			\left(\frac{\theta_2^2\Delta t}{2}-6\theta_3^2(\Delta t L)^2-\frac{6\theta_3^2\Delta t L^2}{\gamma}\right)E_i^{x_i}[|\epsilon^{Z_i}|^2]\nonumber\\
			& \leq \left(1 + \gamma \Delta t +6(1-\theta_3)^2(\Delta t L)^2 +8(1-\theta_1)^2(\Delta t L)^2 + \frac{6(1-\theta_3)^2\Delta t L^2}{\gamma}\right)E_i^{x_i}[|\epsilon^{Y_{i+1}}|^2] \\
			&+ \left(6(1-\theta_3)^2(\Delta t L)^2 +8(1-\theta_1)^2(\Delta t L)^2 +4(\theta_2-1)^2\Delta t+ \frac{6(1-\theta_3)^2\Delta t L^2}{\gamma}\right)E_i^{x_i}[|\epsilon^{Z_{i+1}}|^2] \nonumber\\
			& + 6E_i^{x_i}[|R_{\theta}^{Y_i}|^2] +6E_i^{x_i}[|R_{\mbox{xgb}}^{Y_i}|^2]+ \frac{6E_i^{x_i}[|R_{\theta}^{Y_i}|^2]}{\gamma \Delta t}+\frac{6E_i^{x_i}[|R_{\mbox{xgb}}^{Y_i}|^2]}{\gamma \Delta t} + \frac{4E_i^{x_i}[|R_{\theta}^{Z_i}|^2]}{\Delta t}+4\Delta t\theta_2^2E_i^{x_i}[|R_{\mbox{xgb}}^{Z_i}|^2].\nonumber
		\end{align*}
		We choose $\gamma$ such that
		$\frac{\theta_2^2\Delta t}{2} -\frac{6\theta_3^2\Delta t L^2}{\gamma} \geq 4(\theta_2-1)^2\Delta t+\frac{6(1-\theta_3)^2\Delta t L^2}{\gamma}, i.e. \gamma \geq \frac{12\theta_3^2L^2+12(1-\theta_3)^2L^2}{\theta_2^2-8(\theta_2-1)^2},$ by which the
		latter inequality can be rewritten as
		\begin{align*}
			&\left(1-6\theta_3^2(\Delta t L)^2-\frac{\theta_3^2\Delta t (\theta_2^2-8(\theta_2-1)^2)}{2\theta_3^2+2(1-\theta_3)^2}\right)E_i^{x_i}[|\epsilon^{Y_i}|^2] \\
			&+ \left(\frac{(8\theta_3^2(\theta_2-1)^2+(1-\theta_3)^2\theta_2^2)\Delta t}{2(1-\theta_3)^2+2\theta_3^2} -6\theta_3^2(\Delta t L)^2\right)E_i^{x_i}[|\epsilon^{Z_i}|^2]\nonumber\\
			& \leq \left(1 + \frac{12\theta_3^2L^2+12(1-\theta_3)^2L^2}{\theta_2^2-8(\theta_2-1)^2} \Delta t +6(1-\theta_3)^2(\Delta t L)^2 +8(1-\theta_1)^2(\Delta t L)^2 \right.\\ &\left.+\frac{(1-\theta_3)^2\Delta t (\theta_2^2-8(\theta_2-1)^2)}{2\theta_3^2+2(1-\theta_3)^2}\right)E_i^{x_i}[|\epsilon^{Y_{i+1}}|^2] \\
			&+ \left(\frac{(8\theta_3^2(\theta_2-1)^2+(1-\theta_3)^2\theta_2^2)\Delta t}{2(1-\theta_3)^2+2\theta_3^2}+6(1-\theta_3)^2(\Delta t L)^2 +8(1-\theta_1)^2(\Delta t L)^2 \right)E_i^{x_i}[|\epsilon^{Z_{i+1}}|^2] \nonumber\\
			& + 6E_i^{x_i}[|R_{\theta}^{Y_i}|^2] +6E_i^{x_i}[|R_{\mbox{xgb}}^{Y_i}|^2]+ \frac{(\theta^2_2-8(\theta_2-1)^2)E_i^{x_i}[|R_{\theta}^{Y_i}|^2]}{L^2(2\theta_3^2+2(1-\theta_3)^2) \Delta t}\\
			&+\frac{(\theta^2_2-8(\theta_2-1)^2)E_i^{x_i}[|R_{\mbox{xgb}}^{Y_i}|^2]}{L^2(2\theta_3^2+2(1-\theta_3)^2) \Delta t} + \frac{4E_i^{x_i}[|R_{\theta}^{Z_i}|^2]}{\Delta t}+4\Delta t\theta_2^2E_i^{x_i}[|R_{\mbox{xgb}}^{Z_i}|^2].\nonumber
		\end{align*}
		which implies
		\begin{align*}
			&E_i^{x_i}[|\epsilon^{Y_i}|^2] + \frac{(8\theta_3^2(\theta_2-1)^2+(1-\theta_3)^2\theta_2^2)\Delta t}{2(1-\theta_3)^2+2\theta_3^2}E_i^{x_i}[|\epsilon^{Z_i}|^2] \leq \frac{1+C \Delta t}{1-C \Delta t}
			\left(\left(E_i^{x_i}[|\epsilon^{Y_{i+1}}|^2]\right.\right. \\
			&\left.\left.+ \frac{(8\theta_3^2(\theta_2-1)^2+(1-\theta_3)^2\theta_2^2)\Delta t}{2(1-\theta_3)^2+2\theta_3^2}E_i^{x_i}[|\epsilon^{Z_{i+1}}|^2]\right)\right.+ 6E_i^{x_i}[|R_{\theta}^{Y_i}|^2] +6E_i^{x_i}[|R_{\mbox{xgb}}^{Y_i}|^2]\nonumber\\
			&\left.+ \frac{(\theta^2_2-8(\theta_2-1)^2)E_i^{x_i}[|R_{\theta}^{Y_i}|^2]}{L^2(2\theta_3^2+2(1-\theta_3)^2) \Delta t}+\frac{(\theta^2_2-8(\theta_2-1)^2)E_i^{x_i}[|R_{\mbox{xgb}}^{Y_i}|^2]}{L^2(2\theta_3^2+2(1-\theta_3)^2) \Delta t}+ \frac{4E_i^{x_i}[|R_{\theta}^{Z_i}|^2]}{\Delta t}+4\Delta t\theta_2^2E_i^{x_i}[|R_{\mbox{xgb}}^{Z_i}|^2]\right).\nonumber
		\end{align*}
		By induction, we obtain then
		\begin{align*}
			E_i^{x_i}[|\epsilon^{Y_i}|^2] &+ \frac{(8\theta_3^2(\theta_2-1)^2+(1-\theta_3)^2\theta_2^2)\Delta t}{2(1-\theta_3)^2+2\theta_3^2}E_i^{x_i}[|\epsilon^{Z_i}|^2] \leq \left(\frac{1+C \Delta t}{1-C \Delta t}\right)^{N_T-i}
			\left(E_{N_T-1}^{x_{N_T-1}}[|\epsilon^{Y_{N_T}}|^2] \right. \\&
			\left.+ \frac{(8\theta_3^2(\theta_2-1)^2+(1-\theta_3)^2\theta_2^2)\Delta t}{2(1-\theta_3)^2+2\theta_3^2}E_{N_T-1}^{x_{N_T-1}}[|\epsilon^{Z_{N_T}}|^2]\right)\nonumber\\
			&+\sum_{j=i+1}^{N_T}\left(\frac{1+C \Delta t}{1-C \Delta t}\right)^{j-i}\left(6E_i^{x_i}[|R_{\theta}^{Y_j}|^2] +6E_i^{x_i}[|R_{\mbox{xgb}}^{Y_j}|^2]+ \frac{(\theta^2_2-8(\theta_2-1)^2)E_i^{x_i}[|R_{\theta}^{Y_j}|^2]}{L^2(2\theta_3^2+2(1-\theta_3)^2) \Delta t}\right.\nonumber\\
			&\left.\frac{(\theta^2_2-8(\theta_2-1)^2)E_i^{x_i}[|R_{\mbox{xgb}}^{Y_j}|^2]}{L^2(2\theta_3^2+2(1-\theta_3)^2) \Delta t}+ \frac{4E_i^{x_i}[|R_{\theta}^{Z_j}|^2]}{\Delta t}+4\Delta t\theta_2^2E_i^{x_i}[|R_{\mbox{xgb}}^{Z_j}|^2]\right) \nonumber\\
			&\leq \exp(2CT)
			\left(E_{N_T-1}^{x_{N_T-1}}[|\epsilon^{Y_{N_T}}|^2] + \frac{(8\theta_3^2(\theta_2-1)^2+(1-\theta_3)^2\theta_2^2)\Delta t}{2(1-\theta_3)^2+2\theta_3^2}E_{N_T-1}^{x_{N_T-1}}[|\epsilon^{Z_{N_T}}|^2]\right)\nonumber\\
			&+\exp(2CT)\sum_{j=i+1}^{N_T}\left(6E_i^{x_i}[|R_{\theta}^{Y_j}|^2] +6E_i^{x_i}[|R_{\mbox{xgb}}^{Y_j}|^2]+ \frac{(\theta^2_2-8(\theta_2-1)^2)E_i^{x_i}[|R_{\theta}^{Y_j}|^2]}{L^2(2\theta_3^2+2(1-\theta_3)^2) \Delta t}\right.\nonumber\\
			&\left.\frac{(\theta^2_2-8(\theta_2-1)^2)E_i^{x_i}[|R_{\mbox{xgb}}^{Y_j}|^2]}{L^2(2\theta_3^2+2(1-\theta_3)^2) \Delta t}+ \frac{4E_i^{x_i}[|R_{\theta}^{Z_j}|^2]}{\Delta t}+4\Delta t\theta_2^2E_i^{x_i}[|R_{\mbox{xgb}}^{Z_j}|^2]\right) .\nonumber
		\end{align*}
The regression error $R^{Z}_{\mbox{xgb}}$ and $R^{Y}_{\mbox{xgb}}$ are given in \eqref{eq:err_z1} and \eqref{eq:error_y1}, from which one can deduce e.g., $|R_{\mbox{xgb}}^{Y_j}|^2\leq |2(Var^{y}_j + \hat{\mathcal{L}}_{\mbox{min}}(\hat{q}^{y}_j))|^2:=(\mbox{Var}^{\mathcal{Y}}_j)^2.$ Similarly, for $\frac{|R_{\mbox{xgb}}^{Y_j}|^2}{\Delta t}$ and $\Delta t|R_{\mbox{xgb}}^{Z_j}|^2$ we obtain $\frac{N_T(\mbox{Var}_j^{\mathcal{Y}})^2}{T}$ and $\frac{T(\mbox{Var}_j^{\mathcal{Z}})^2}{N_T},$ respectively. Finally, with the known conditions and bounds of the local errors mentioned above we complete the proof.
\end{proof}
Note that one can straightforwardly obtain
\begin{equation*}
	E_0^{x_0}\left[|\epsilon^{Y_{i}}|^2 + \frac{9\Delta t}{16}|\epsilon^{Z_{i}}|^2\right]\leq Q (\Delta t)^{4} + \tilde{Q}\sum_{i+1}^{N_T}\left(\frac{N_T (\mbox{Var}^{\mathcal{Y}}_j)^2}{T}+\frac{T(\mbox{Var}^{\mathcal{Z}}_j)^2}{N_T}\right),
\end{equation*}
when $\theta_1=\theta_2=\theta_3=1/2$ provided that 
$E_{N_T-1}^{x_{N_T-1}}[|\epsilon^{Z_{N_T}}|^4] \thicksim \mathcal{O}((\Delta t)^4)~\mbox{and}~ E_{N_T-1}^{x_{N_T-1}}[|\epsilon^{Y_{N_T}}|^4] \thicksim \mathcal{O}((\Delta t)^{4}).$

\section{Numerical experiments}\label{sec:numexp}
In this section we use some numerical examples to show the accuracy of our methods for solving the high-dimensional ($\geq 100$ dim) BSDEs. As already introduced above, $N_T~\mbox{and}~M$ are the total discrete time steps and sampling size, respectively. For all the examples, we consider an equidistant time grid and perform $10$ Picard iterations. We ran the algorithms $10$ times independently and take average value of absolute error, whereas the different
seeds are used for each simulation. Numerical experiments were performed with an Intel(R) Core(TM) i5-8500 CPU @ 3.00GHz and 15 GB RAM.

\subsection{The less challenging problems}\label{sec:bsdeexample}
If the values of the driver function $f$ are almost constant or behave linearly  along $\mathbb{X}_t=(X_t, Y_t, Z_t)$, in particular when $dX_t=dW_t,$ i.e., a standard BSDE, very well approximations $(\hat{Y}_0,\hat{Z}_0),$ 
can be reached by averaging the samples at $T$ generated with $x_0,$ i.e., by using the Monte-Carlo estimation. A fine time-discretization and regression are not really necessary. If $g$ is differentiable, according to Scheme 1 we use 
\begin{align}
\hat{Z}_{0}&\approx \frac{1}{M} \sum_{\mathcal{M}=1}^{M}\left[\frac{2}{T}\hat{y}_{N_T, \mathcal{M}}W_{N_T,\mathcal{M}} + f(T,\hat{x}_{N_T, \mathcal{M}}, \hat{y}_{N_T, \mathcal{M}}, \hat{z}_{N_T, \mathcal{M}})W_{N_T,\mathcal{M}}-\hat{z}_{N_T, \mathcal{M}}\right],\label{eq:shortscheme1}\\
\hat{Y}_{0}&\approx \frac{1}{M} \sum_{\mathcal{M}=1}^{M}\left[\hat{y}_{N_T, \mathcal{M}} + \frac{T}{2} f(T,\hat{x}_{N_T, \mathcal{M}}, \hat{y}_{N_T, \mathcal{M}}, \hat{z}_{N_T, \mathcal{M}})\right] + \frac{T}{2} f(0,x_0, \hat{Y}_0, \hat{Z}_{0}),\label{eq:shortscheme2}
\end{align}
where $\hat{y}_{N_T, \mathcal{M}} = g(\hat{x}_{N_T, \mathcal{M}}),\,\hat{z}_{N_T, \mathcal{M}}=g_x(\hat{x}_{N_T, \mathcal{M}}).$ Similarly, if $g$ is not differentiable, one can average according to Scheme 2.

\paragraph{Example 1}
We consider firstly a BSDE with quadratically growing derivatives derived in \cite{Gobet2015}, whose explicit solution is known. A modified version of that BSDE in 100-dimensional case is analyzed numerically in \cite{E2017}, and given by
	\begin{equation*}
	-dY_t = \underbrace{\norm{Z}^2_{\mathbb{R}^{1 \times d}}-\norm{\nabla\psi(t, W_t)}^2_{\mathbb{R}^{d}}-(\partial_t+\frac{1}{2}\Delta)\psi(t,W_t)}_{=f}\,dt - Z_t\,dW_t
	\end{equation*}	 
	with the analytical solution
	\begin{equation*}
	\left\{
	\begin{array}{l}
	Y_t = \psi(t,W_t)=\sin\left((T-t+\frac{1}{d}\norm{W_t}^2_{\mathbb{R}^{d}})^{\alpha}\right),\\  
	Z_t =2\alpha W_t^{\top} \cos\left((T-t+\frac{1}{d}\norm{W_t}^2_{\mathbb{R}^{d}})^{\alpha}\right)(T-t+\frac{1}{d}\norm{W_t}^2_{\mathbb{R}^{d}})^{\alpha-1},
	\end{array}\right.
	\end{equation*}
where $\alpha \in (0, 1/2],$ we let $\alpha=0.4.$ We obverse that the driver $f$ behaves almost linearly, see $f(T,\hat{x}_{N_T, \mathcal{M}}, \hat{y}_{N_T, \mathcal{M}}, \hat{z}_{N_T, \mathcal{M}})$ displayed in Figure \ref{fig:linearf}. This is to say that we should be able use \eqref{eq:shortscheme1} and \eqref{eq:shortscheme2}. 
\begin{figure}[htbp!]
	\centering
	\begin{subfigure}[b]{0.4\textwidth}
		\includegraphics[width=\textwidth]{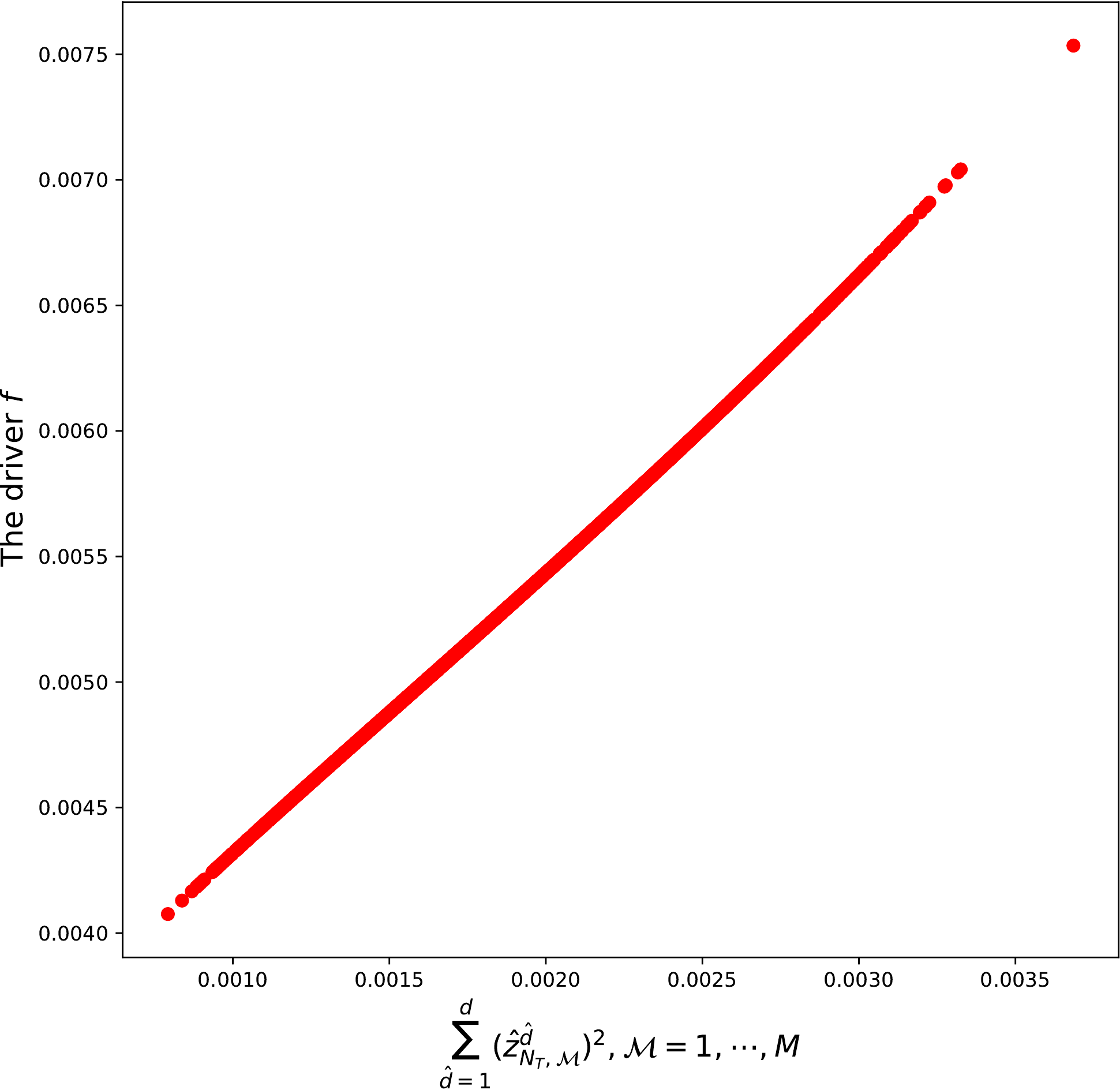}
		\subcaption{$T=1,\,d=100.$}
	\end{subfigure}
	~ 
	\begin{subfigure}[b]{0.4\textwidth}
		\includegraphics[width=\textwidth]{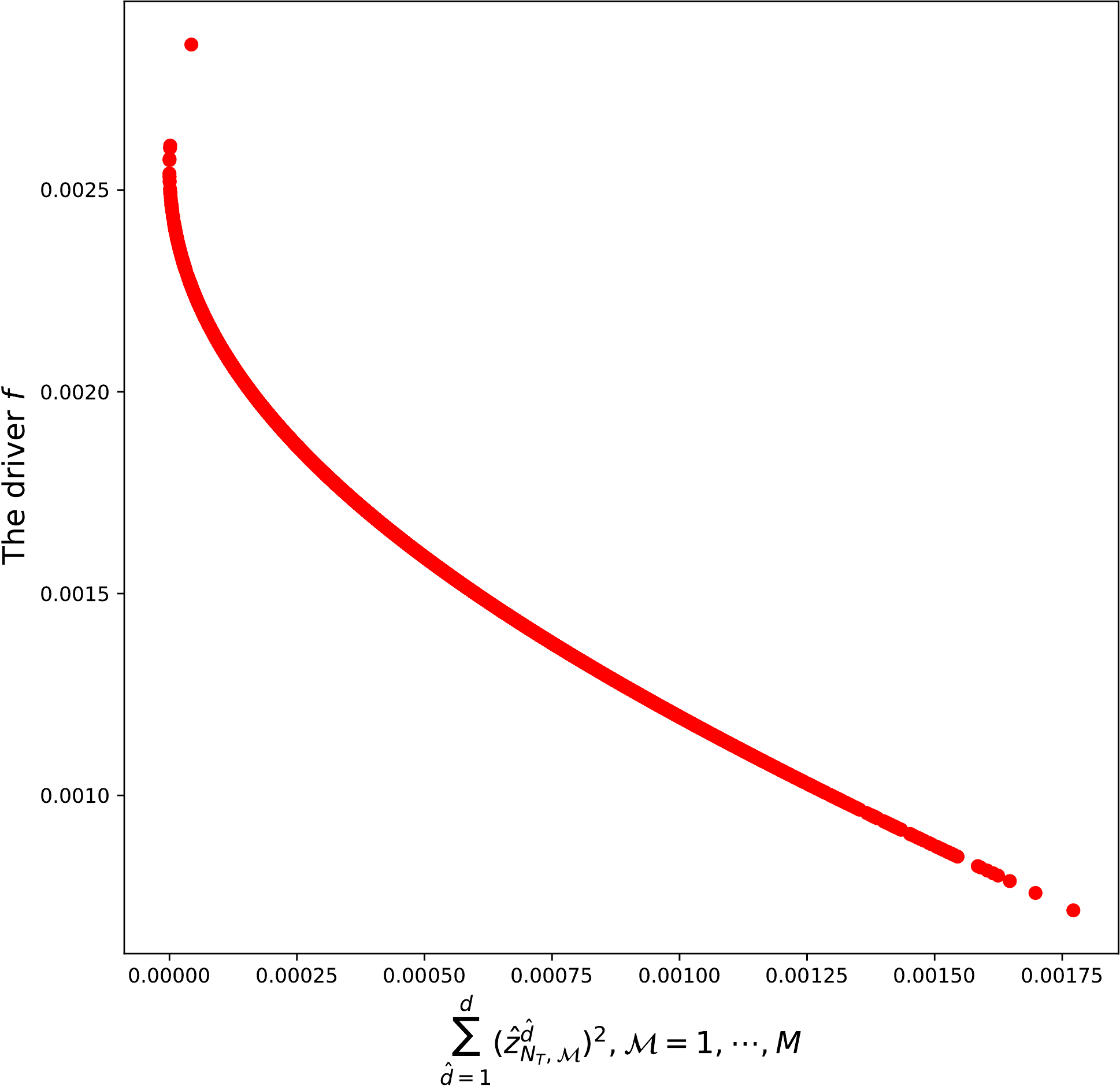}
		\subcaption{$T=5,\,d=100.$}
	\end{subfigure}\\
\begin{subfigure}[b]{0.39\textwidth}
	\includegraphics[width=\textwidth]{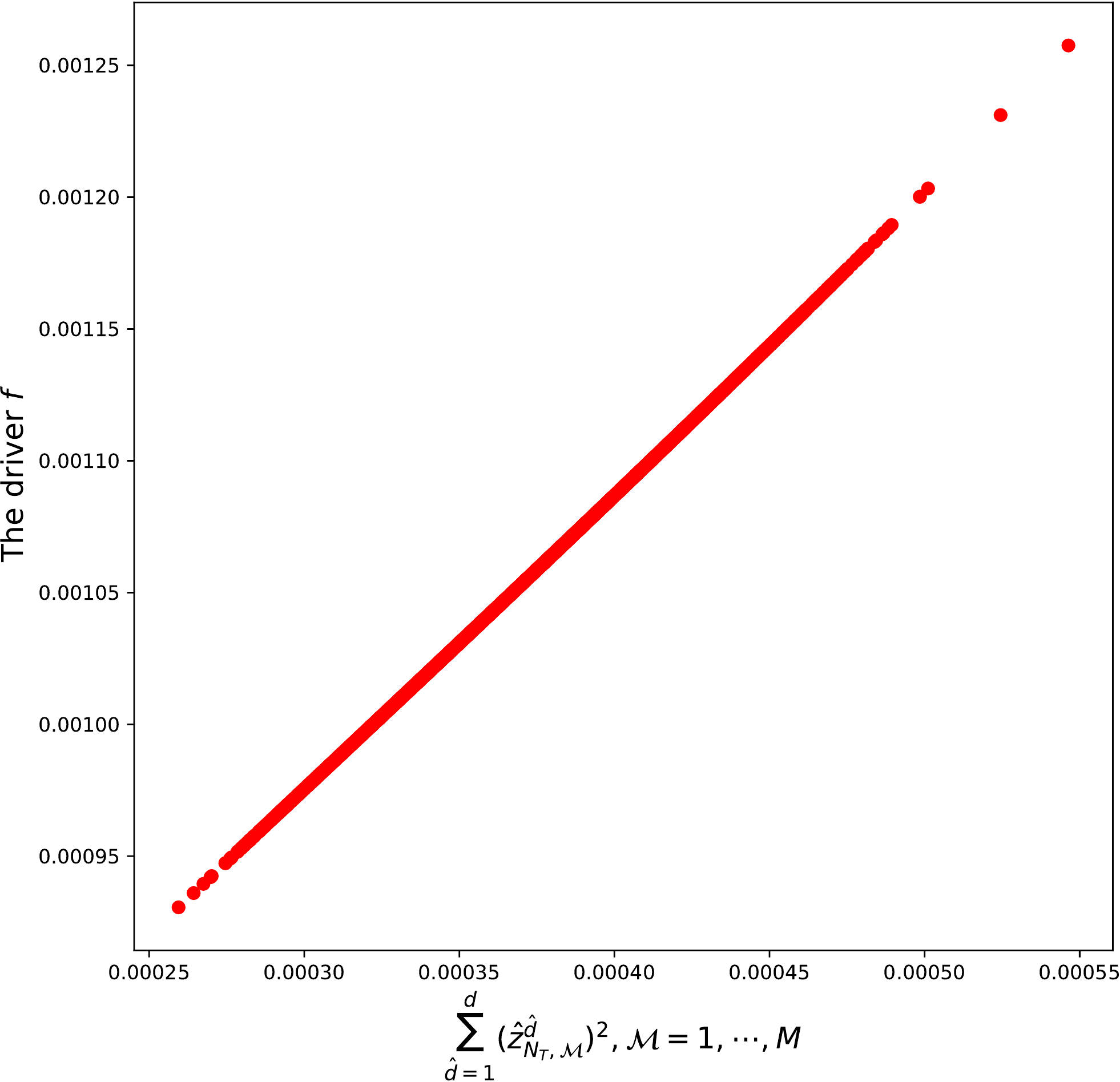}
	\subcaption{$T=1,\,d=500.$}
\end{subfigure}
~
\begin{subfigure}[b]{0.39\textwidth}
	\includegraphics[width=\textwidth]{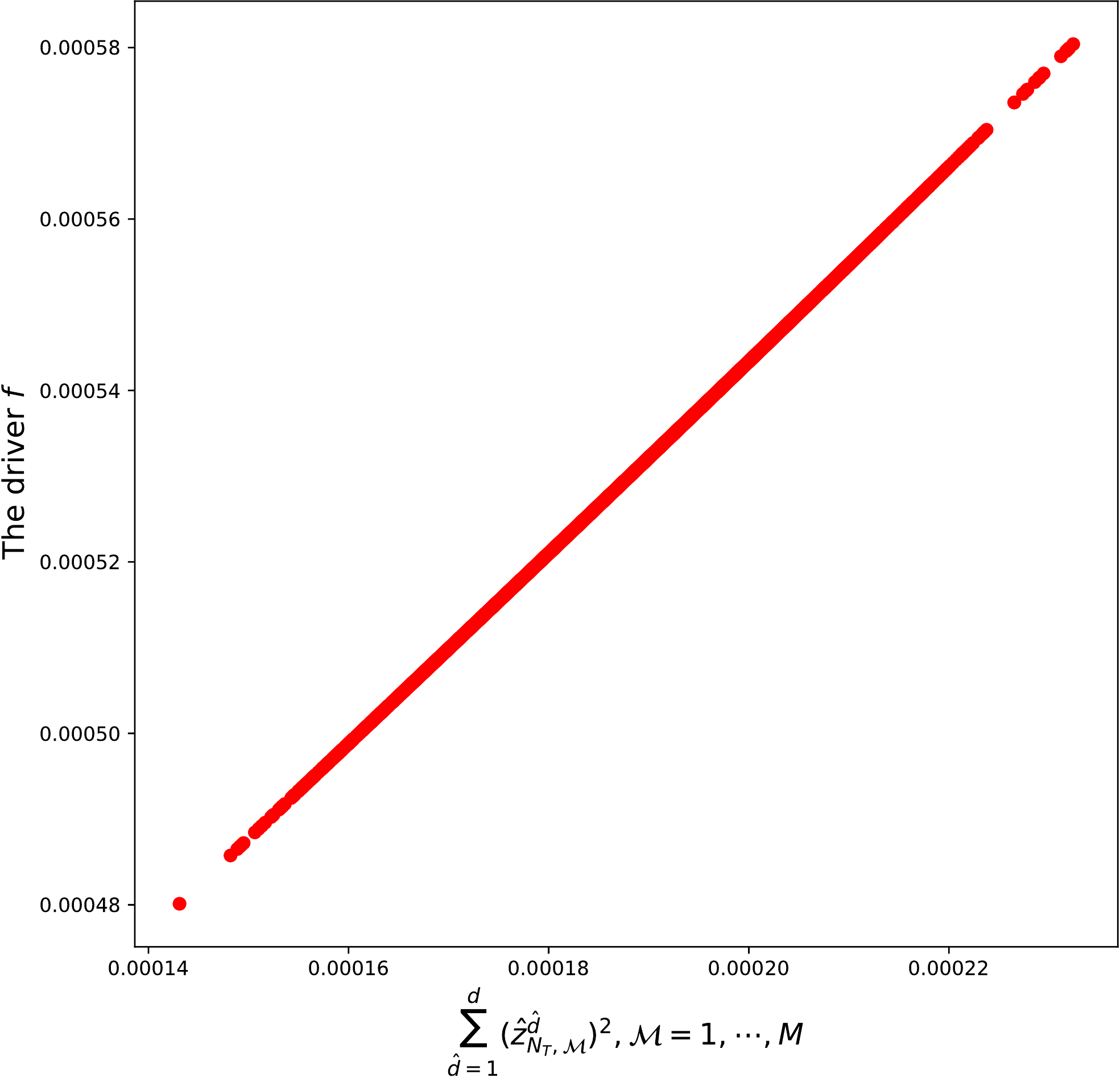}
	\subcaption{$T=1,\,d=1000.$}
\end{subfigure}
	\caption{The samples of $f(T, X, Y, Z)$ for the samples of $\norm{Z}^2_{\mathbb{R}^{1 \times d}}$ using the different values of $d, T$ for Example 1. }\label{fig:linearf}
\end{figure}
We test that with different values for $T$ and $d$ and report our results in Table \ref{table:01}, and the average runtime of each run(in seconds) is provided as well. In high-dimensional case we have $\mathbf{\hat{Z}_0}=(\hat{Z}_0^1, \hat{Z}_0^2, \cdots, \hat{Z}_0^d),$  let $\hat{Y}_{0,k}$ and $\mathbf{\hat{Z}_{0,k}}$ denote the result on the $k$-th run of the algorithm, $k=1,\cdots, 10,$ while $(Y_0, \mathbf{Z_0})$ is used for the exact solution or reference value. In our tests we consider average of the absolute errors, i.e., $error_y:=\frac{1}{10}\sum_{k=1}^{10} |Y_0-\hat{Y}_{0,k}|$ and $error_z:=\frac{1}{10}\sum_{k=1}^{10} \frac{\sum^d_{\hat{d}=1} |Z_0^{\hat{d}}-\hat{Z}_{0,k}^{\hat{d}}|}{d},$ as well as the empirical standard deviations $\sqrt{\frac{1}{9}\sum_{k=1}^{10}|\hat{Y}_{0,k} - \hat{Y}_0|^2}$ and $\sqrt{\frac{1}{9}\sum_{k=1}^{10}\left|\frac{\sum^d_{\hat{d}=1} \hat{Z}_{0,k}^{\hat{d}}}{d}-\hat{\overline{Z}}_{0}\right|^2}$ with $\hat{Y}_0=\frac{1}{10}\sum_{k=1}^{10}\hat{Y}_{0,k},$ $\hat{\overline{Z}}_{0}=\frac{1}{10}\sum_{k=1}^{10}\frac{\sum^d_{\hat{d}=1} \hat{Z}_{0,k}^{\hat{d}}}{d}.$
\begin{table}
	\centering 
	\small
	\begin{tabular}{|c|c|c|c|c|c|c|c|{c}r}
		\hline
\multirow{2}{*}{$d$}&  Theoretical  & $M=10000$ & $M=20000$ & $M=50000$ & $M=100000$   \\
&solution & $error_y$(Std. dev.) & $error_y$(Std. dev.) & $error_y$(Std. dev.) & $error_y$(Std. dev.)\\
\multirow{2}{*}{$T$}	& $Y_0$ & $error_z$(Std. dev.) & $error_z$(Std. dev.) & $error_z$(Std. dev.) & $error_z$(Std. dev.)\\
& ${\mathbf Z_0}$ & avg. runtime & avg. runtime & avg. runtime & avg. runtime\\
		\hline
		\multirow{2}{*}{$100$} 
	    & $0.84147$ & $0.01475(0.00177)$ & $0.00712(0.00131)$ & $0.00280(0.00049)$ & $0.00139(0.00023)$   \\
		& $\mathbf{0}_d$ & $0.01381(0.00093)$ & $0.00957(0.00091)$ & $0.00596(0.00048)$ & $0.00428(0.00029)$   \\
		$1$&& $0.67$ & $1.43$ & $3.72$ & $7.39$   \\
		\hline
		\multirow{2}{*}{$100$} &
		$0.96859$ & $0.01959(0.00238)$ & $0.00950(0.00167)$ & $0.00371(0.00056)$ & $0.00188(0.00029)$   \\
		& $\mathbf{0}_d$ & $0.01123(0.00076)$ & $0.00777(0.00073)$ & $0.00484(0.00039)$ & $0.00347(0.00023)$   \\
		$2$&& $0.69$ & $1.49$ & $3.76$ & $7.26$   \\
		\hline
		\multirow{2}{*}{$100$} &  $0.99982$ & $0.02092(0.00258)$ & $0.01017(0.00170)$ & $0.00394(0.00053)$ & $0.00204(0.00031)$   \\
		& $\mathbf{0}_d$& $0.00947(0.00062)$ & $0.00654(0.00061)$ & $0.00407(0.00033)$ & $0.00292(0.00019)$   \\
		$3$&  & $0.69$ & $1.47$ & $3.73$ & $7.22$   \\
		\hline
		\multirow{2}{*}{$100$} & $0.98553$ & $0.02039(0.00257)$ & $0.00994(0.00159)$ & $0.00383(0.00047)$ & $0.00202(0.00031)$   \\
		& $\mathbf{0}_d$ & $0.00808(0.00054)$ & $0.00557(0.00052)$ & $0.00347(0.00029)$ & $0.00248(0.00016)$   \\
		$4$&  & $0.69$ & $1.47$ & $3.79$ & $7.28$   \\
		\hline
		\multirow{2}{*}{$100$} 
		& $0.94511$ & $0.01883(0.00245)$ & $0.00921(0.00143)$ & $0.00353(0.00044)$ & $0.00190(0.00030)$   \\
		& $\mathbf{0}_d$ & $0.00694(0.00045)$ & $0.00478(0.00045)$ & $0.00297(0.00025)$ & $0.00213(0.00014)$   \\
		$5$& & $0.67$ & $1.47$ & $3.74$ & $7.21$   \\
		\hline
		\multirow{2}{*}{$500$} 
	& $0.84147$ & $0.07103(0.00412)$ & $0.03465(0.00162)$ & $0.01426(0.00080)$ & $0.00704(0.00027)$   \\
	& $\mathbf{0}_d$ & $0.01347(0.00040)$ & $0.00941(0.00021)$ & $0.00605(0.00018)$ & $0.00424(0.00008)$   \\
	$1$&  & $12.65$ & $25.60$ & $62.33$ & $115.49$   \\
	\hline
		\multirow{2}{*}{$1000$} 
	& $0.84147$ & $0.14018(0.00629)$ & $0.07058(0.00249)$ & $0.02788(0.00097)$ & $0.01406(0.00026)$   \\
	& $\mathbf{0}_d$ & $0.01336(0.00031)$ & $0.00945(0.00017)$ & $0.00597(0.00010)$ & $0.00423(0.00005)$   \\
	$1$&  & $46.06$ & $89.93$ & $218.29$ & $433.45$   \\
	\hline
		\end{tabular}
\caption{Numerical simulation using the method \eqref{eq:shortscheme1} and \eqref{eq:shortscheme2} for Example 1.}\label{table:01}
\end{table}
We see that the approximations are very impressive.

\paragraph{Example 2} Another high-dimensional example considered in the recent literature is the time-dependent reaction-diffusion-type equation
	\begin{equation*}
-dY_t = \min\left\{1, \left[Y_t-\kappa-1-\sin\left(\zeta\sum_{\hat{d}=1}^d W_t^{\hat{d}}\right)\exp\left(\frac{\zeta^2d(t-T)}{2}\right)\right]^2 \right\}\,dt - Z_t\,dW_t
\end{equation*}	 
with the analytical solution
\begin{equation*}
\left\{
\begin{array}{l}
Y_t =1+\kappa+\sin\left(\zeta\sum_{\hat{d}=1}^d W_t^{\hat{d}}\right)\exp\left(\frac{\zeta^2d(t-T)}{2}\right),\\  
Z_t =\zeta\cos\left(\zeta\sum_{\hat{d}=1}^d W_t^{\hat{d}}\right)\exp\left(\frac{\zeta^2d(t-T)}{2}\right)\mathbf{1}_{d},
\end{array}\right.
\end{equation*}
which is oscillating.
\begin{figure}[htbp!]
	\centering
	\begin{subfigure}[b]{0.33\textwidth}
		\includegraphics[width=\textwidth]{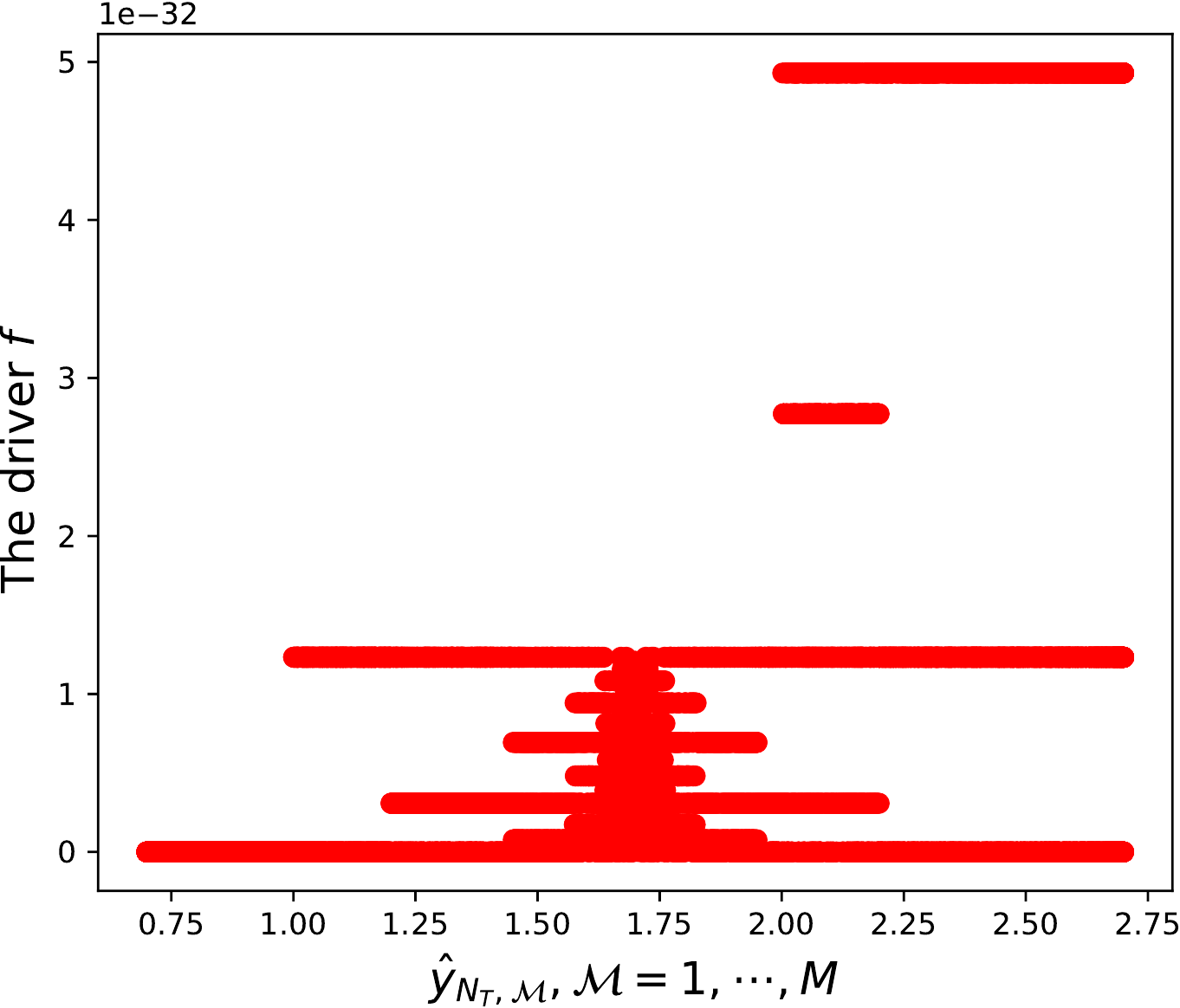}
		\subcaption{$T=1,\,d=100.$}
	\end{subfigure}
	~ 
	\begin{subfigure}[b]{0.33\textwidth}
		\includegraphics[width=\textwidth]{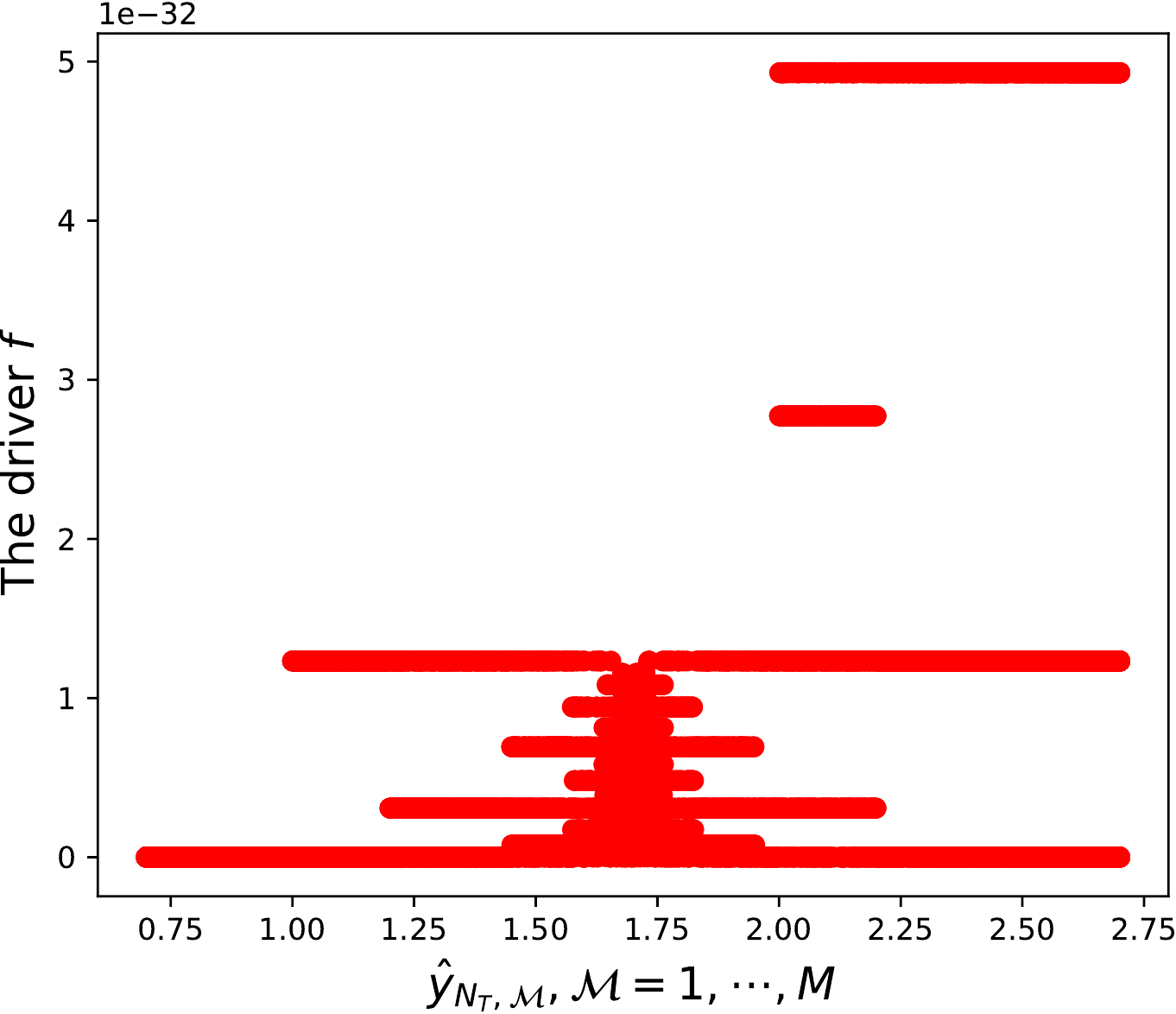}
		\subcaption{$T=5,\,d=100.$}
	\end{subfigure}\\
	\begin{subfigure}[b]{0.33\textwidth}
		\includegraphics[width=\textwidth]{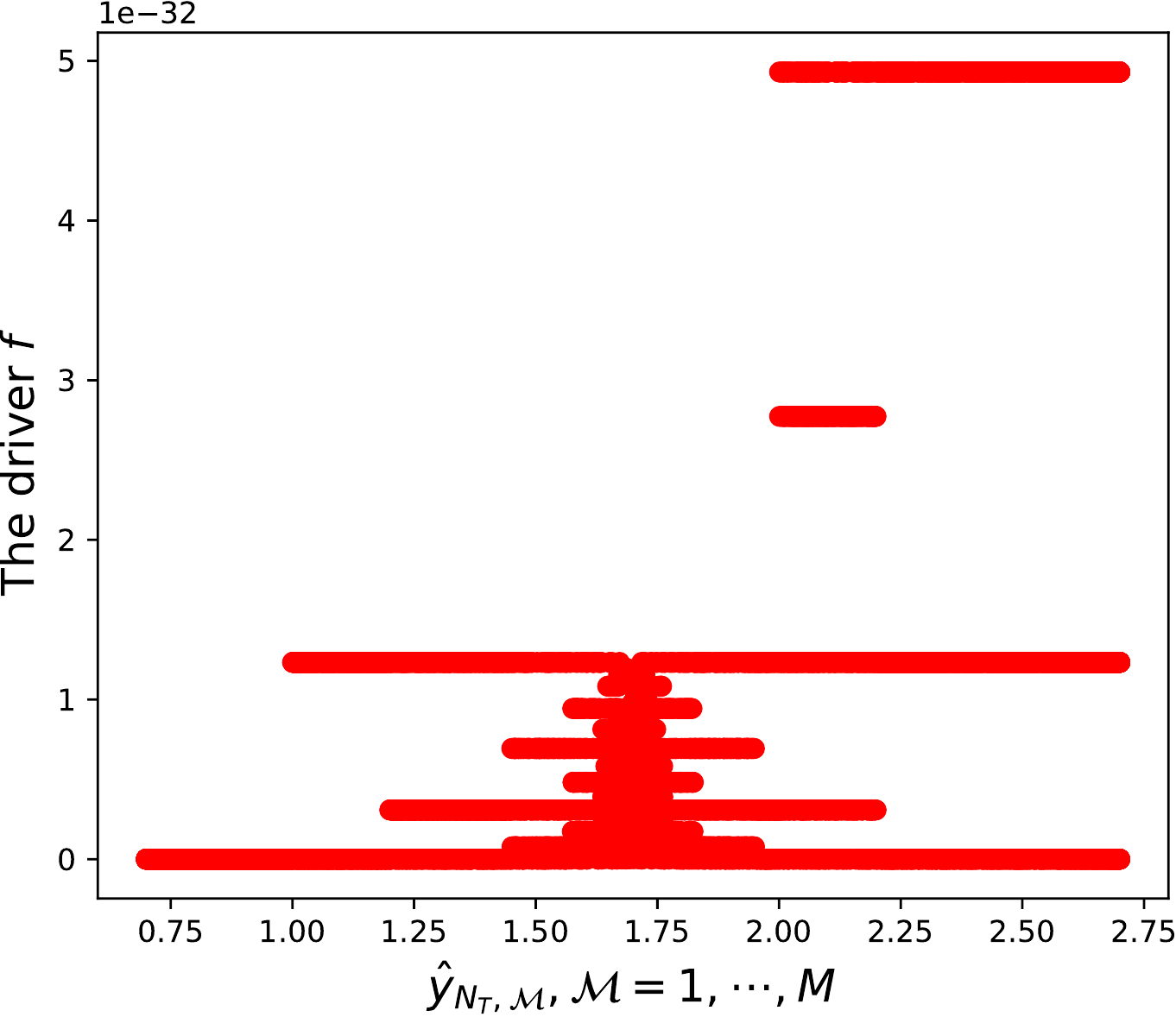}
		\subcaption{$T=1,\,d=500.$}
	\end{subfigure}
	~
	\begin{subfigure}[b]{0.33\textwidth}
		\includegraphics[width=\textwidth]{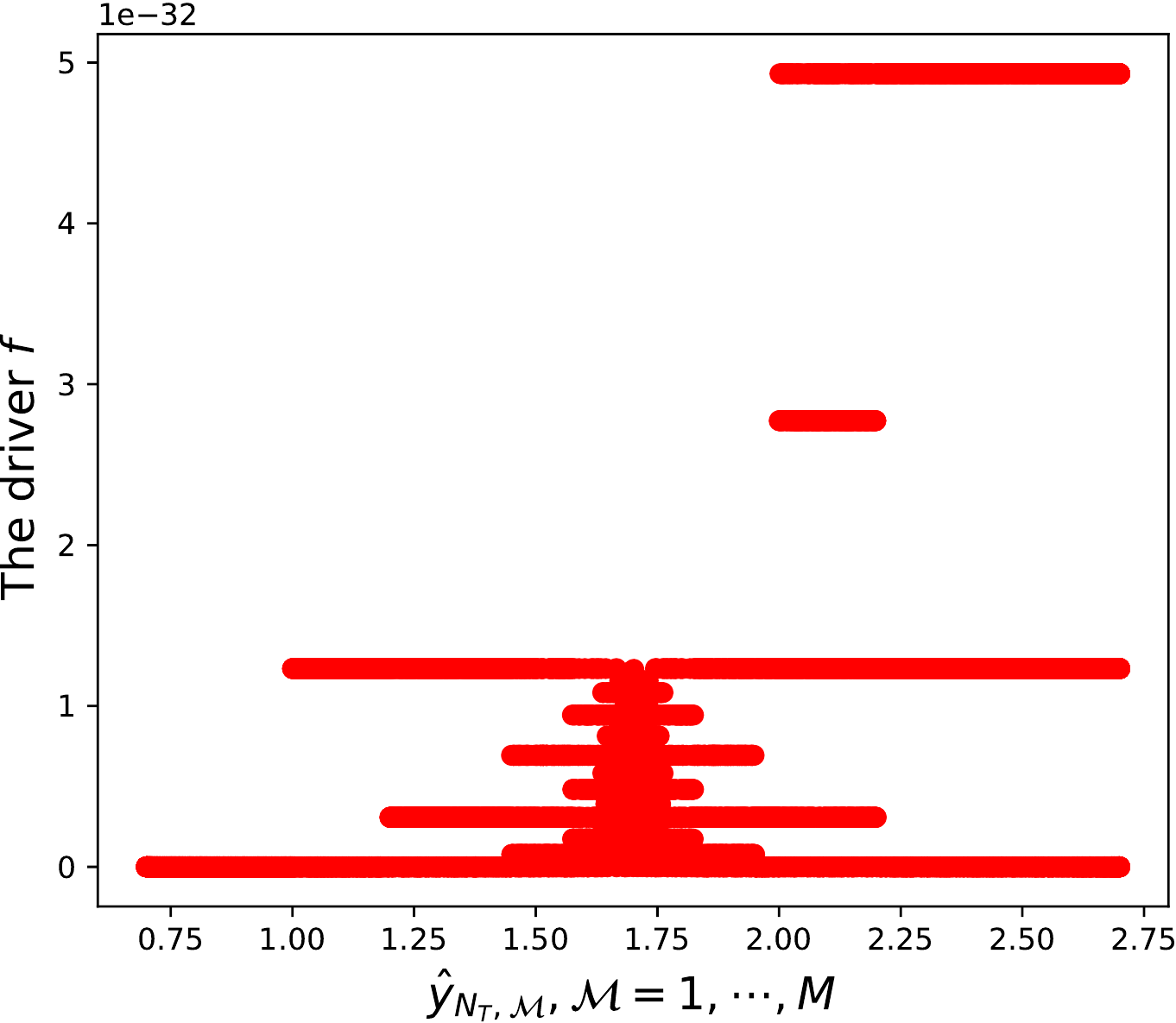}
		\subcaption{$T=1,\,d=1000.$}
	\end{subfigure}
	\caption{The samples of $f(T, X, Y, Z)$ for the samples $Y$ using the different values of $d, T$ for Example 2.}\label{fig:constantf}
\end{figure}
This example has been numerically analyzed in \cite{Gobet2017} for $d=2,$ and in \cite{E2017} for $d=100.$ In our test we find that driver function gives very small values (up to $10^{-31}$), see Figure \eqref{fig:constantf}, i.e., \eqref{eq:shortscheme1} and \eqref{eq:shortscheme2} can be used. Let $\kappa=\frac{7}{10},$ $\zeta=\frac{1}{\sqrt{d}},$ we report our results in Table \ref{table:02}.
\begin{table}
	\centering 
	\small
	\begin{tabular}{|c|c|c|c|c|c|c|{c}r}
		\hline
		\multirow{2}{*}{$d$}&  Theoretical  & $M=10000$ & $M=20000$ & $M=50000$ & $M=100000$   \\
		&solution & $error_y$(Std. dev.) & $error_y$(Std. dev.) & $error_y$(Std. dev.) & $error_y$(Std. dev.)\\
		\multirow{2}{*}{$T$}	& $Y_0$ & $error_z$(Std. dev.) & $error_z$(Std. dev.) & $error_z$(Std. dev.) & $error_z$(Std. dev.)\\
				& ${\mathbf Z_0}$ & avg. runtime & avg. runtime & avg. runtime & avg. runtime\\
		\hline
		\multirow{2}{*}{$100$} & $1.7$
 & $0.00724(0.00875)$ & $0.00419(0.00569)$ & $0.00243(0.00309)$ & $0.00176(0.00222)$   \\
		&$0.06065 \,\mathbf{1}_d$& $0.00381(0.00459)$ & $0.00213(0.00262)$ & $0.00144(0.00176)$ & $0.00091(0.00120)$   \\
		\multirow{1}{*}{$1$}& &$0.14$ & $0.32$ & $0.80$ & $1.66$   \\
		\hline
		\multirow{2}{*}{$100$} & 
		 $1.7$& $0.00629(0.00769)$ & $0.00425(0.00579)$ & $0.00215(0.00266)$ &$0.00167(0.00205)$   \\
		& $0.03679 \,\mathbf{1}_d$ & $0.00269(0.00295)$ & $0.00150(0.00185)$ & $0.00101(0.00125)$ & $0.00071(0.00092)$   \\
	\multirow{1}{*}{$2$}	&& $0.15$ & $0.33$ & $0.85$ & $1.67$   \\
		\hline
		\multirow{2}{*}{$100$} & 
		 $1.7$ & $0.00499(0.00682)$ & $0.00405(0.00545)$ & $ 0.00157(0.00210)$ & $0.00169(0.00195)$   \\
		& $0.02231 \,\mathbf{1}_d$ & $ 0.00324(0.00192)$ & $0.00144(0.00139)$ & $0.00080(0.00104)$ & $0.00065(0.00079)$   \\
	\multirow{1}{*}{$3$}	& & $0.17$ & $0.35$ & $0.87$ & $1.67$   \\
		\hline
		\multirow{2}{*}{$100$} & 
		 $1.7$ & $ 0.00538(0.00677)$ & $0.00377(0.00502)$ & $0.00151(0.00182)$ & $0.00177(0.00206)$   \\
		& $0.01353 \,\mathbf{1}_d$ & $0.00508(0.00147)$ & $0.00196(0.00104)$ & $0.00072(0.00085)$ & $0.00056(0.00068)$   \\
	\multirow{1}{*}{$4$}	&  & $0.17$ & $0.34$ & $0.87$ & $1.68$   \\
		\hline
		\multirow{2}{*}{$100$} &
		 $1.7$ & $0.00557(0.00706)$ & $0.00345(0.00465)$ & $0.00157(0.00185)$ & $0.00185(0.00219)$   \\
		& $0.00821 \,\mathbf{1}_d$ & $0.00670(0.00129)$ & $0.00339(0.00084)$ & $0.00098(0.00059)$ & $0.00045(0.00054)$   \\
	\multirow{1}{*}{$5$}	&  & $0.17$ & $0.37$ & $0.89$ & $1.69$   \\
		\hline
		\multirow{2}{*}{$500$} & 
		 $1.7$ & $0.00537(0.00739)$ & $0.00383(0.00447)$ & $0.00243(0.00301)$ & $0.00189(0.00220)$   \\
		& $0.02712 \,\mathbf{1}_d$ & $0.00921(0.00103)$ & $0.00359(0.00096)$ & $0.00083(0.00088)$ & $0.00047(0.00053)$   \\
	\multirow{1}{*}{$1$}	&  & $2.72$ & $5.52$ & $13.72$ & $27.10$   \\
		\hline
		\multirow{2}{*}{$1000$} & $1.7$ & $0.00623(0.00765)$ & $0.00444(0.00539)$ & $0.00266(0.00299)$ & $0.00187(0.00234)$   \\
		& $0.01918 \,\mathbf{1}_d$ & $0.01341(0.00081)$ & $0.00680(0.00060)$ & $0.00179(0.00047)$ & $0.00041(0.00043)$   \\
	\multirow{1}{*}{$1$}	&  & $10.17$ & $20.40$ & $50.94$ & $101.17$   \\
		\hline
	\end{tabular}
	\caption{Numerical simulation using the method \eqref{eq:shortscheme1} and \eqref{eq:shortscheme2} for Example 2.}\label{table:02}
\end{table}

\subsection{General nonlinear high-dimensional problems}\label{sec:bs}
In our methods, the most important thing is to find the right values for the XGBoost hyperparameters to prevent overfitting and underfitting. In principle, one can run {\it GridSearchCV} to find best values of the hyperparameters, however, this is quite time consuming. Therefore, in our experiments we tune the parameters separately with the following remarks.
\begin{itemize}
\item In our test the results are not really sensitive with respect to the maximum depth of a tree, we fix thus $\tilde{d}$ to be $2$ for less computational cost in all the following examples.
\item The datasets are splitted into train and test sets with a ratio of $75:25.$ 
\item We find that the most important parameters are the learning rate and the number of trees, namely $K.$ In our test we can obtain promising results with any values of learning rate in the set of $\{0.01, 0.1, 0.2, \cdots, 0.9, 0.99\}$ by adjusting a proper value of $K.$ We denote the number of trees in individual XGBoost regressor at each time step for computing $Z^{\Delta_t}_i$ and $Y^{\Delta_t}_i$ by $K_{i,z}$ and $K_{i,y},$ $i=0,\cdots, N_T-1,$ respectively. In principle, we can adjust values of $K_{i,z}$ and $K_{i,y},$ i.e., at each time step. However, this is quite time consuming and thus maybe not realistic. Fortunately, we observe the learning curves for $K_{i,z}$ and $K_{i,y}$ behave quite similarly for different time step and $\Delta t.$ Therefore, for all the time steps we consider $K_{z}$ and $K_{y}$ for computing $Z$- and $Y$- component, respectively. In Example 3 and 4 we fix the learning rate to be $0.9$ for a faster computation, and choose the proper numbers of trees, namely $K_z$ and $K_y$ by comparing the training and test MSEs. For the challenging problems, i.e., Example 5 and 6 we fix the learning rate to be $0.1$ and then correspondingly select proper values for $K_z$ and $K_y.$
\item For all other parameters we use the default values, e.g., $\lambda=1$ and $\gamma=0.$
\end{itemize}
For each example we perform $10$ independent runs.
We denote the approximations with XGBoost regressions by $(Y^{\Delta t}_0,\mathbf{Z}^{\Delta t}_0)$ with $\mathbf{{Z}^{\Delta t}_0}=(Z_0^{\Delta t,1}, Z_0^{\Delta t,2}, \cdots, Z_0^{\Delta t,d}).$ For the $Y$-component we define the error and standard deviation as: $error_y:=\frac{1}{10}\sum_{k=1}^{10} |Y_0-Y^{\Delta t}_{0,k}|$ and  $\sqrt{\frac{1}{9}\sum_{k=1}^{10}|Y^{\Delta t}_{0,k} - Y^{\Delta t}_0|^2}$ with
$Y_0^{\Delta t}=\frac{1}{10}\sum_{k=1}^{10}Y_{0,k}^{\Delta t}.$
Furthermore, for the $Z$-component we consider 
$error_z:=\frac{1}{10}\sum_{k=1}^{10} \frac{\sum^d_{\hat{d}=1} |Z_0^{\hat{d}}-Z_{0,k}^{\Delta t, \hat{d}}|}{d}$
 and $\sqrt{\frac{1}{9}\sum_{k=1}^{10}\left|\frac{\sum^d_{\hat{d}=1} Z_{0,k}^{\Delta t, \hat{d}}}{d}-\overline{Z}^{\Delta t}_{0}\right|^2}$ with  $\overline{Z}^{\Delta t}_{0}=\frac{1}{10}\sum_{k=1}^{10}\frac{\sum^d_{\hat{d}=1} Z_{0,k}^{\hat{d},\Delta t}}{d}.$
\paragraph{Example 3} To test our Scheme 2 we consider a pricing problem of an European option in a financial market with different interest rate for borrowing and lending to hedge the option. This pricing problem is analyzed in \cite{Bergman2008}, used as a standard nonlinear (high-dimensional)
example in the many works, see e.g., \cite{Bender2017, E2017, E2019, Gobet2005, Kapllani2020, Teng2019}, and given by
\begin{equation*}
	\begin{split}
		\left\{
		\begin{array}{rcl}
			dS_{t} &=& \mu S_{t}\,dt + \sigma S_{t} dW_{t},\\ 
			-dY_t & = &-R^l Y_t - \frac{\mu-R^l}{\sigma} \sum_{\hat{d}=1}^d Z^{\hat{d}}_t + (R^b-R^l)\max\left(0, \frac{1}{\sigma}\sum_{\hat{d}=1}^d Z_t^{\hat{d}} - Y_t\right)\,dt -Z_t \,dW_t,\\  
			Y_T&=&\max\left(\max_{\hat{d}=1,\cdots, d}(S_{T}^{\hat{d}})-K_1, 0\right) - 2 \max\left(\max_{\hat{d}=1,\cdots, d}(S_{T}^{\hat{d}})-K_2, 0\right)
		\end{array}
		\right. 
	\end{split}
\end{equation*}
where $\sigma>0,$ $\mu \in \mathbb{R},$ $R^b, R^l$ are different interest rates and $K_1, K_2$ are strikes. Since $Z_{N_T}$ is not analytically available in this example, we choose Scheme 2. The parameter values are set as: $T=0.5,$ $\mu=0.06,$ $\sigma=0.02,$ $R^l=0.04,$ $R^b=0.06, K_1=120$ and $K_2=150,$ for which the reference price $Y_0=21.2988$ is computed using the multilevel Monte Carlo with $7$ Picard iterations \cite{E2019}. As mentioned above, we set the learning rate to be $0.9$ and $\tilde{d}=2$ for a faster computation, and tune separately to choose the value of $K_z$ and $K_y.$ We show the XGBoost model for Y with learning curves in Figure \ref{fig:01}. 
In Figure \ref{fig:01a} we see that overfitting occurs for a large value of $K_y.$ We set thus $K_y=20$ by observing the learning curves in Figure \ref{fig:01b}, which are the enlargement of the curves in Figure \ref{fig:01a} until $K_y=100.$ Similarly, the value of $K_z$ can be tuned as well, we use $K_z=20$ in this example.

\begin{figure}[htbp!]
	\centering
	\begin{subfigure}[b]{0.47\textwidth}
			\includegraphics[width=\textwidth]{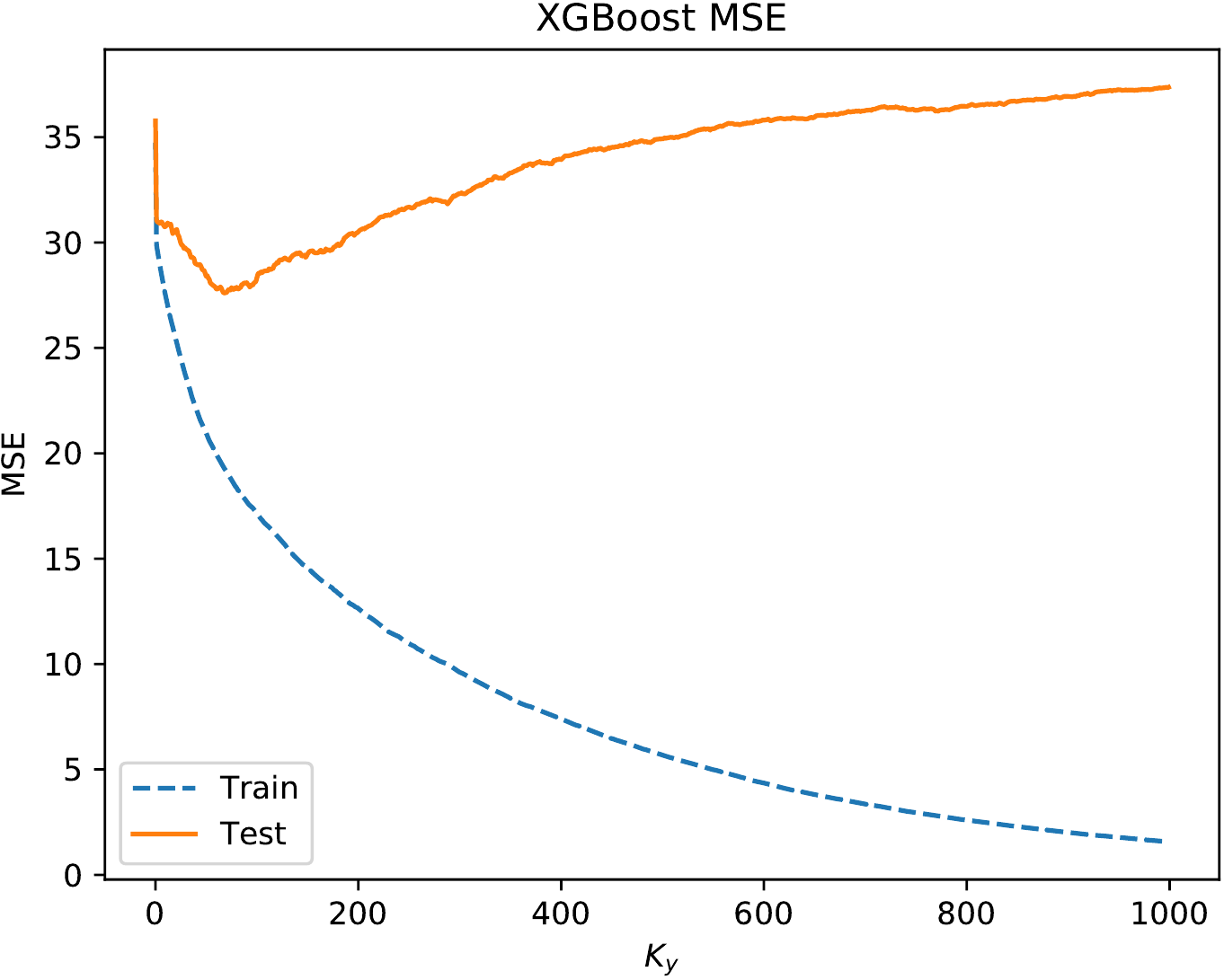}
		\subcaption{The XGBoost model for $Y$ until $K_y=1000.$\newline}\label{fig:01a}
	\end{subfigure}
	~ 
	\begin{subfigure}[b]{0.47\textwidth}
		\includegraphics[width=\textwidth]{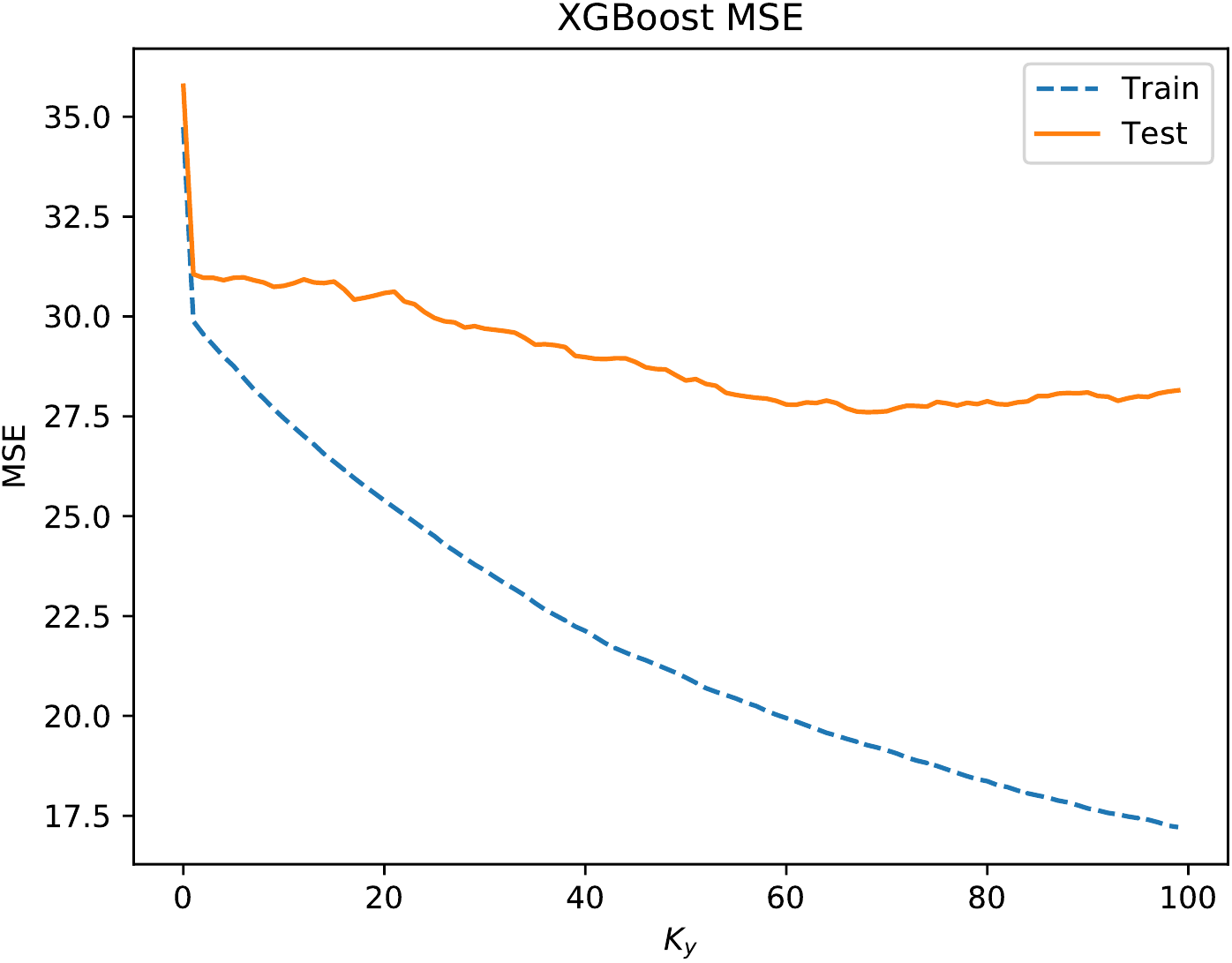}
		\subcaption{The enlargement of the learning curves in (a) until $K_y=100.$}\label{fig:01b}
	\end{subfigure}
	\caption{The MSEs of the XGBoost models in Scheme 2 for Example 3 for different numbers of trees, $N_T=10, M=10000$ and the learning rate is $0.9.$}\label{fig:01}
\end{figure}
The numerical results are reported in Table \ref{table:03}.
\begin{table}
	\centering 
	\small
	\begin{tabular}{|c|c|c|c|c|c|c|{c}r}
		\hline
		\multirow{1}{*}{$d=100$}&  Ref. value computed & $M=10000$ & $M=20000$ & $M=50000$ & $M=100000$   \\
		\multirow{1}{*}{$T=0.5$}&with \cite{E2019} & $error_y$(Std. dev.) & $error_y$(Std. dev.) & $error_y$(Std. dev.) & $error_y$(Std. dev.)\\
		\cline{1-1} \multirow{1}{*}{$N_T$}& \multirow{1}{*}{$Y_0$} &  avg. runtime & avg. runtime  & avg. runtime & avg. runtime\\
		\hline
		\multirow{3}{*}{$10$} 
		& \multirow{3}{*}{$21.2988$} & \multirow{2}{*}{$0.13725(0.13335)$} & \multirow{2}{*}{$0.13911(0.09088)$} & \multirow{2}{*}{$0.12952(0.06001)$} & \multirow{2}{*}{$0.18669(0.04351)$}   \\
		 && &  & &  \\
		 & & $9.10$ & $19.37$ & $54.06$ & $131.17$   \\
		\hline
		\multirow{3}{*}{$20$} 
		& \multirow{3}{*}{$21.2988$} &\multirow{2}{*}{ $0.20207(0.21176)$} & \multirow{2}{*}{$0.14609(0.16716)$} & \multirow{2}{*}{$0.05542(0.03960)$} & \multirow{2}{*}{$0.08281(0.01219)$}   \\
 	&& &  & &  \\
	&& $25.03$ & $51.73$ & $139.14$ & $324.49$   \\
		\hline
		\multirow{3}{*}{$30$} 
		& \multirow{3}{*}{$21.2988$} & \multirow{2}{*}{$0.33619(0.43693)$} & \multirow{2}{*}{$0.14689(0.15127)$} & \multirow{2}{*}{$0.04741(0.05735)$} & \multirow{2}{*}{$0.04096(0.05090)$}   \\
 	&& &  & &  \\
	&  & $40.94$ & $84.17$ & $224.13$ & $519.47$\\
		\hline
	\end{tabular}
	\caption{Numerical simulation using Scheme 2 for Example 3.}\label{table:03}
\end{table}
Note that the same reference price is used to compare the deep learning-based numerical methods for high-dimensional BSDEs in \cite{E2017} (Table 3), 
which has achieved a relative error of $0.0039$ in a runtime of $566$ seconds. From Table \ref{table:03} one see that the relative errors $0.00222$ and $0.00192$ can be achieved in runtime $224.13$ and $519.47,$ respectively.
 \paragraph{Example 4} In \cite{Beck2019b}, several examples have been numerically analyzed up to $10000$ dimensions. Depending on complexity of the solution structure, the computational expenses are widely different, where the least computational effort is shown for computing the Allen-Cahn equation
\begin{equation*}
	\begin{split}
		\left\{
		\begin{array}{rcl}
			dX_t & = &\sigma \, dW_t,\\ 
			-dY_t & = & \left( Y_t - Y_t^3 \right)\,dt -Z_t \,dW_t,\\  
			Y_T &=& \arctan\left(\max_{\hat{d}\in\{1,2,\cdots,d\}} X_T^{\hat{d}}\right). 
		\end{array}
		\right. 
	\end{split}
\end{equation*}
In this example, one has a cubic nonlinearity, and $Z_T$ can be analytically calculated as $\left(0, \cdots,\frac{1}{1+\left( X_T^{d_{m}}\right)^2},\cdots, 0\right),$ where $d_m$ denotes the index of maximum value. This is to say that we can use Scheme 1, for a comparative purpose we select parameter values as those in \cite{Beck2019b}: $T=0.3,~\sigma=\sqrt{2}$ and $N_T=10.$ For the XGBoost hyperparameters we use the same values as those in Example 3. In Table \ref{table:04} we present the numerical results for different values for $d$ and $M.$
\begin{table}
	\centering 
	\begin{tabular}{|c|c|c|c|c|c|{c}r}
		\hline
		\multirow{1}{*}{$T=0.3,~N_T=10$}&\multirow{3}{*}{Ref. value}& $M=2000$ & $M=5000$   \\
		\cline{1-1} && $error_y$(Std. dev.) & $error_y$(Std. dev.)   \\
	    \multirow{1}{*}{$d$}&& avg. runtime & avg. runtime   \\
		\hline
		\multirow{2}{*}{$10$} 
		&\multirow{2}{*}{$0.89060$}&   $0.00279(0.00342)$ & $0.00175(0.00233)$    \\
		& & $0.19$ & $0.33$    \\
		\hline
		\multirow{2}{*}{$50$} 
		&\multirow{2}{*}{$1.01830$} & $0.00141(0.00187)$ & $0.00076(0.00076)$    \\
		& & $0.45$ & $1.18$    \\
		\hline
		\multirow{2}{*}{$100$} 
		&\multirow{2}{*}{$1.04510$}& $0.00265(0.00147)$ & $0.00098(0.00113)$    \\
		& & $0.85$ & $2.21$    \\
		\hline
		\multirow{2}{*}{$200$} 
		&\multirow{2}{*}{$1.06220$}& $0.00101(0.00130)$ & $0.00074(0.00097)$    \\
		& & $1.69$ & $4.31$    \\
		\hline
		\multirow{2}{*}{$300$} 
		&\multirow{2}{*}{$1.07217$}& $0.00247(0.00171)$ & $0.00075(0.00044)$    \\
		&& $2.53$ & $6.74$    \\
		\hline
		\multirow{2}{*}{$500$} 
		&\multirow{2}{*}{$1.08124$} & $0.00134(0.00110)$ & $0.00071(0.00034)$    \\
		&& $4.37$ & $11.79$    \\
		\hline
		\multirow{2}{*}{$1000$} 
		&\multirow{2}{*}{$1.09100$} & $0.00111(0.00142)$ & $0.00051(0.00103)$    \\
		&& $9.25$ & $25.33$    \\
		\hline
		\multirow{2}{*}{$5000$} 
		&\multirow{2}{*}{$1.10691$}& $0.00162(0.00086)$ & $0.00174(0.00012)$    \\
		&& $69.51$ & $129.90$    \\
		\hline
		\multirow{2}{*}{$10000$} 
		&\multirow{2}{*}{$1.11402$} & $0.00049(0.00087)$ & $0.00037(0.00017)$    \\
		& & $151.89$ & $670.24$    \\
		\hline
	\end{tabular}
	\caption{Numerical simulation using Scheme 1 for Example 4.}\label{table:04}
\end{table}
We see that substantially less data are required for very well approximations in this example, and we obtain a better accuracy than that achieved in \cite{Beck2019b} for less computational cost.
\paragraph{Example 5}
To test our Scheme 1 exhaustively, we consider the Burgers-type equation
\begin{equation*}
\begin{split}
\left\{
\begin{array}{rcl}
dX_t & = &\sigma \, dW_t,\\ 
-dY_t & = & \left( Y_t - \frac{2+d}{2d} \right) \left( \sum_{\hat{d}=1}^{d} Z_t^{\hat{d}}\right)\,dt -Z_t \,dW_t,\\  
\end{array}
\right. 
\end{split}
\end{equation*}
with the analytic solution
\begin{equation*}
\begin{split}
\left\{
\begin{array}{rcl}
Y_t & = & \frac{\exp\left( t +\frac{1}{d} \sum_{\hat{d}=1}^{d} X_t^{\hat{d}}\right)}{1+\exp\left( t +\frac{1}{d} \sum_{\hat{d}=1}^{d} X_t^{\hat{d}}\right)},\\
Z_t & = & \frac{\sigma}{d} \frac{\exp\left( t +\frac{1}{d} \sum_{\hat{d}=1}^{d} X_t^{\hat{d}}\right)}{\left(1+\exp\left( t +\frac{1}{d} \sum_{\hat{d}=1}^{d} X_t^{\hat{d}}\right)\right)^2}\mathbf{1}_{d}.
\end{array}
\right. 
\end{split}
\end{equation*}
This example has been analyzed in \cite{Chassagneux2014} for $d=3,~T=1$ and $\sigma=1,$ and in \cite{E2019} for $d=100,~T=0.5,~\sigma=0.25$ as well as in \cite{E2017} for $d=20,~T=1,~\sigma=\frac{d}{\sqrt{2}}$ and $d=50,~T=0.2,~\sigma=\frac{d}{\sqrt{2}}.$ This problem in high-dimensional case is computationally challenging, the deep-learning based algorithm in \cite{E2017} seems diverges for $d=100,~T=0.5,~\sigma=\frac{d}{\sqrt{2}}$ (at least based on our attempts). Furthermore, the approximations of $Z$ in the case of $d=100$ are not given in \cite{E2019} and \cite{E2017}. 

Here, we solve this problem for $d=100,~T=0.5,~\sigma=\frac{d}{\sqrt{2}}$ numerically using Scheme 1. We display the MSEs of the XGBoost models for $K_z$ and $K_y$ on the training and test datasets in a time step in Figure \ref{fig:02}, where $N_T=10, M=10000.$ 
\begin{figure}[htbp!]
	\centering
	\begin{subfigure}[b]{0.47\textwidth}
		\includegraphics[width=\textwidth]{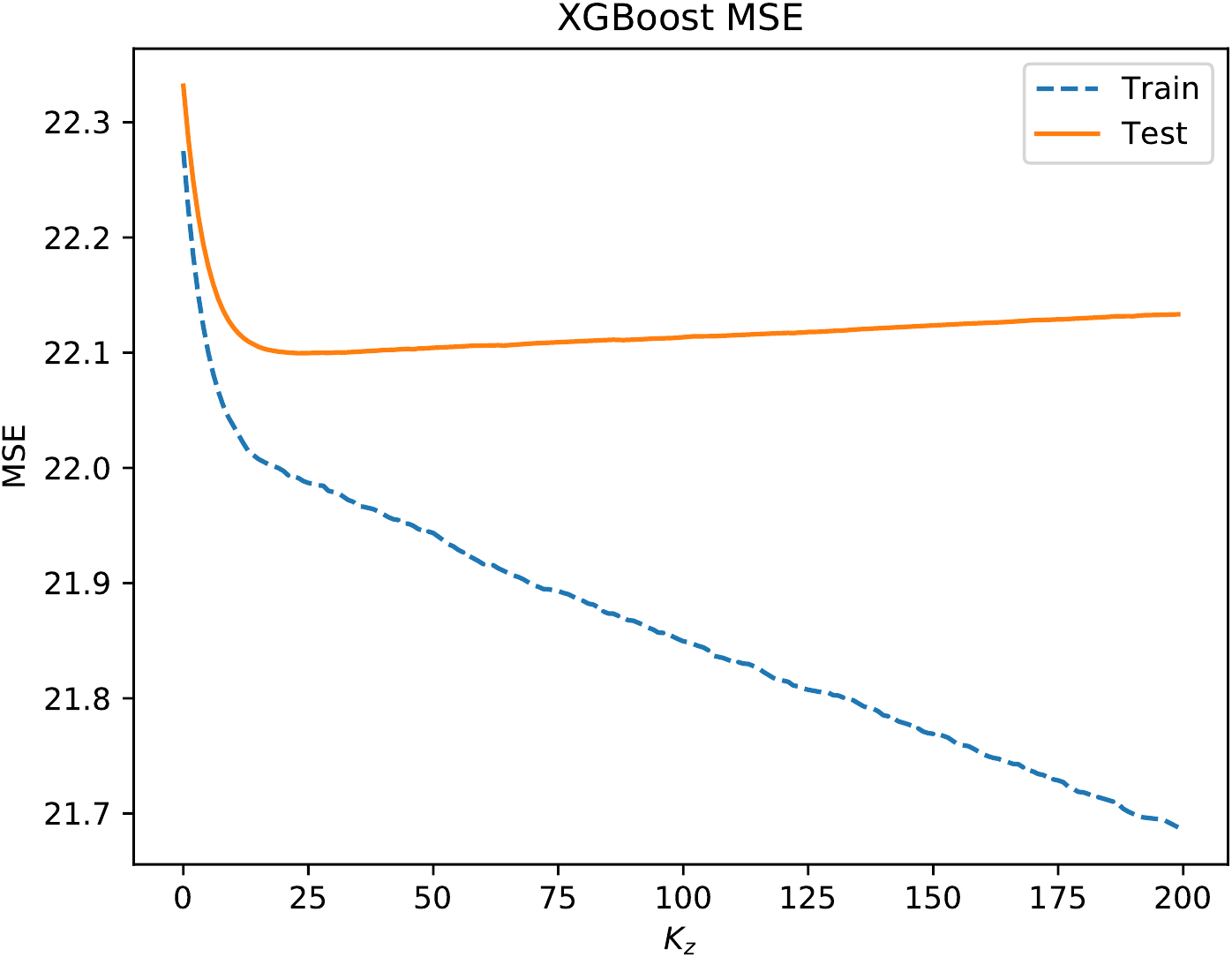}
		\subcaption{The XGBoost model for $Z.$}
	\end{subfigure}
	~ 
	\begin{subfigure}[b]{0.47\textwidth}
		\includegraphics[width=\textwidth]{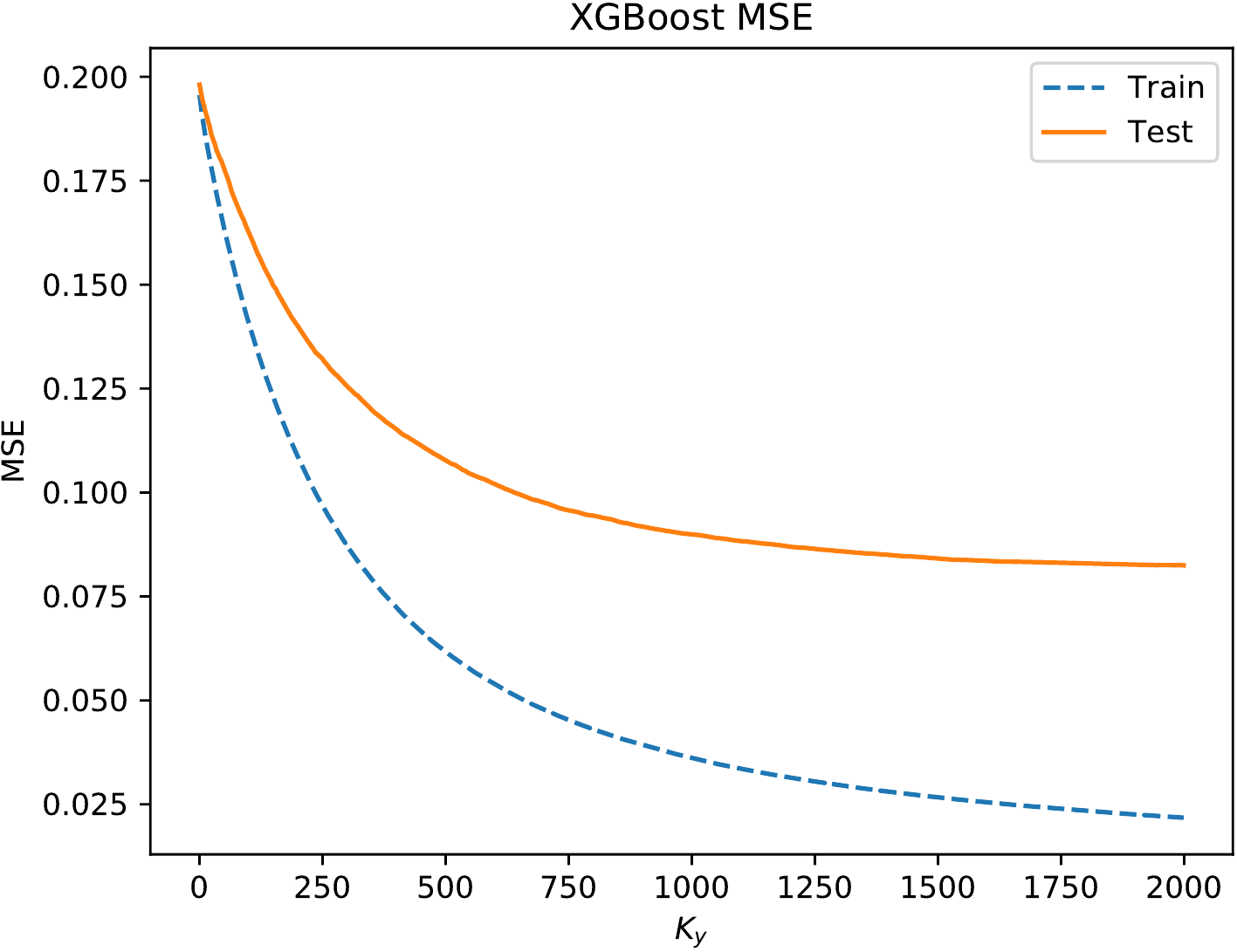}
		\subcaption{The XGBoost model for $Y.$}
	\end{subfigure}
	\caption{The MSEs of the XGBoost models in Scheme 1 for Example 5 for different numbers of trees, $N_T=10, M=10000$ and the learning rate is $0.1.$ }\label{fig:02}
\end{figure}
From Figure \ref{fig:02}, it looks like that the right values of $K_z$ and $K_y$ are around $6,$ we thus let $K=6$ in this example. Note that the training error for $Z$ can be further reduced (near zero) for a large value of $K_z,$ however, the overfitting becomes thus more severe. Finally, we present our approximations in Table \ref{table:05} for different values of $M, N_T.$
\begin{table}
	\centering 
	\small
	\begin{tabular}{|c|c|c|c|c|c|c|{c}r}
		\hline
		\multirow{1}{*}{$d=100$}&  Theoretical  & $M=10000$ & $M=20000$ & $M=50000$ & $M=100000$   \\
	\multirow{1}{*}{$T=0.5$}&solution & $error_y$(Std. dev.) & $error_y$(Std. dev.) & $error_y$(Std. dev.) & $error_y$(Std. dev.)\\
		\cline{1-1}& $Y_0$ & $error_z$(Std. dev.) & $error_z$(Std. dev.) & $error_z$(Std. dev.) & $error_z$(Std. dev.)\\
	\multirow{1}{*}{$N_T$}& ${\mathbf Z_0}$ & avg. runtime & avg. runtime & avg. runtime & avg. runtime\\
		\hline
		\multirow{3}{*}{$10$} 
		& $0.5$ & $0.05486(0.03434)$ & $0.05570(0.02464)$ & $0.05545(0.01359)$ & $0.05203(0.01430)$   \\
& $0.17678\,\mathbf{1}_d$ & $0.00601(0.00412)$ & $0.00529(0.00505)$ & $0.00460(0.00244)$ & $0.00454(0.00135)$   \\
&  & $10.66$ & $22.23$ & $59.58$ & $152.96$   \\
		\hline
		\multirow{3}{*}{$20$} 
		& $0.5$ & $0.01625(0.00038)$ & $0.01629(0.00019)$ & $0.01650(0.00010)$ & $0.01640(0.00009)$   \\
& $0.17678\,\mathbf{1}_d$ & $0.00641(0.00881)$ & $0.00560(0.00629)$ & $0.00454(0.00381)$ & $0.00387(0.00235)$   \\
& & $27.11$ & $55.67$ & $146.40$ & $369.11$   \\
		\hline
		\multirow{3}{*}{$30$} 
		& $0.5$& $0.00712(0.00010)$ & $0.00712(0.00005)$ & $0.00714(0.00005)$ & $0.00713(0.00003)$   \\
	& $0.17678\,\mathbf{1}_d$ & $0.00785(0.00494)$ & $0.00526(0.00509)$ & $0.00519(0.00259)$ & $0.00424(0.00289)$   \\
	&  & $43.55$ & $88.39$ & $234.21$ & $583.45$   \\
		\hline
	\end{tabular}
	\caption{Numerical simulation using Scheme 1 for Example 5.}\label{table:05}
\end{table}
We see that the numerical results are surprisingly very good, it looks like that one needs a larger value of $M$ to balance the discretization error for computing $Z$ than $Y.$ 
\paragraph{Example 6 (A challenging problem)}
To further test our proposed Scheme 1 we consider a BSDE with an unbounded and complex structure solution, which has been analyzed in \cite{Chassagneux2021, Hure2020} and reads
\begin{equation*}
	\begin{split}
		\left\{
		\begin{array}{rcl}
			dX_t & = & \frac{1}{\sqrt{d}} \mbox{I}_d \, dW_t,\\ 
			-dY_t & = & \left( 1 + \frac{T-t}{2d}\right)A(X_t)+B(X_t)+ C\cos\left( \sum_{\hat{d}=1}^{d} \hat{d}\,Z^{\hat{d}}\right)\,dt -Z_t \,dW_t,\\  
		\end{array}
		\right. 
	\end{split}
\end{equation*}
with 
\begin{equation*}
A(x)=\frac{1}{d}\sum_{\hat{d}=1}^d\sin(x^{\hat{d}})\mathds{1}_{\{x^{\hat{d}}<0\}},\,B(x)=\frac{1}{d}\sum_{\hat{d}=1}^d x^{\hat{d}}\mathds{1}_{\{x^{\hat{d}}\geq 0\}},\,C = \frac{(d+1)(2d+1)}{12},
\end{equation*}
and the analytic solution
\begin{equation*}
Y_t  =  \frac{T-t}{d}\sum_{\hat{d}=1}^{d}\left(\sin(X^{\hat{d}}_t)\mathds{1}_{\{X^{\hat{d}}_t<0\}}+X^{\hat{d}}_t\mathds{1}_{\{X^{\hat{d}}_t\geq 0\}}\right)+\cos\left( \sum_{\hat{d}=1}^{d} \hat{d}\,Z_{\hat{d}}\right).
			\end{equation*}
In Table \ref{table:06} we report firstly our approximations for the different values of $M, N_T$ when $d=1, 2, 5.$ Our proposed scheme works very well for the challenging problem. 
			\begin{table}
				\centering 
				\small
				\begin{tabular}{|c|c|c|c|c|c|}
					\hline
					\multirow{3}{*}{$T=1$}& $M=10000$ & $M=50000$ & $M=100000$ & $M=200000$   \\
					& $error_y$(Std. dev.) & $error_y$(Std. dev.) & $error_y$(Std. dev.) & $error_y$(Std. dev.)\\
					& avg. runtime & avg. runtime & avg. runtime & avg. runtime\\
					\hline
					\multirow{1}{*}{$N_T$}&\multicolumn{4}{|c|}{$d=1,$ $Y_0=1.3776, K_z=10, K_y=100$}\\
					\hline
					\multirow{2}{*}{$10$} 
					& $0.00371(0.00501)$ & $0.00153(0.00188)$ & $0.00096(0.00112)$ & $0.00118(0.00169)$   \\
					& $0.72$ & $2.94$ & $5.80$ & $11.62$   \\
					\hline
					\multirow{2}{*}{$20$} 
					& $0.00565(0.00649)$ & $0.00127(0.00121)$ & $0.00173(0.00220)$ & $0.00087(0.00128)$   \\
					& $1.51$ & $6.25$ & $12.36$ & $24.77$   \\
					\hline
					\multirow{2}{*}{$30$} 
					& $0.00526(0.00598)$ & $0.00112(0.00170)$ & $0.00159(0.00191)$ & $0.00134(0.00141)$   \\
					& $2.31$ & $9.54$ & $18.86$ & $37.98$   \\
					\hline
					\multirow{1}{*}{$N_T$}&\multicolumn{4}{|c|}{$d=2,$ $Y_0=0.5707, K_z=8, K_y=150$}\\
					\hline
					\multirow{2}{*}{$10$} 
					& $0.00893(0.01154)$ & $0.00505(0.00622)$ & $0.00278(0.00359)$ & $0.00258(0.00337)$   \\
					& $1.10$ & $4.69$ & $9.21$ & $18.69$   \\
					\hline
					\multirow{2}{*}{$20$} 
					& $0.01156(0.01370)$ & $0.00376(0.00437)$ & $0.00317(0.00386)$ & $0.00327(0.00340)$   \\
					& $2.31$ & $10.01$ & $19.81$ &  $40.05$  \\
					\hline
					\multirow{2}{*}{$30$} 
					& $0.01167(0.01772)$ & $0.00558(0.00607)$ & $0.00325(0.00425)$ & $0.00177(0.00252)$   \\
					& $3.52$ & $15.32$ & $30.34$ & $61.56$   \\
					\hline
					\multirow{1}{*}{$N_T$}&\multicolumn{4}{|c|}{$d=5,$ $Y_0 =0.8466, K_z=2, K_y=150$}\\
					\hline
					\multirow{2}{*}{$10$} 
					& $0.02626(0.03105)$ & $0.01533(0.01038)$ & $0.01191(0.00681)$ & $0.00917(0.00545)$   \\
					& $1.68$ & $7.91$ & $16.02$ & $32.79$   \\
					\hline
					\multirow{2}{*}{$20$} 
					& $0.01854(0.02541)$ & $0.01101(0.01310)$ & $0.00537(0.00761)$ & $0.00398(0.00489)$   \\
					& $3.58$ & $17.27$ & $34.72$ & $70.96$   \\
					\hline
					\multirow{2}{*}{$30$} 
					& $0.02439(0.03115)$ & $0.00687(0.00947)$ & $0.00718(0.01015)$ & $0.00452(0.00437)$   \\
					& $5.48$ & $26.49$ & $53.30$ & $108.78$   \\
					\hline
				\end{tabular}
				\caption{Numerical simulation using Scheme 1 for Example 6.}\label{table:06}
			\end{table}
Note that the reported values of $K_z$ and $K_y$ in Table \ref{table:06} are optional. For example, we display the MSEs of the XGBoost models for $K_z$ and $K_y$ on the training and test datasets in a time step in Figure \ref{fig:03} for $d=1,$ from which we roughly choose $K_z=8$ and $K_y=100.$ In our tests, the almost same results can be obtained with values of $K_z$ in $\{2,3,\cdots,25\}$ and $K_y$ in $\{10,11,\cdots,200\}.$
\begin{figure}[htbp!]
	\centering
	\begin{subfigure}[b]{0.47\textwidth}
		\includegraphics[width=\textwidth]{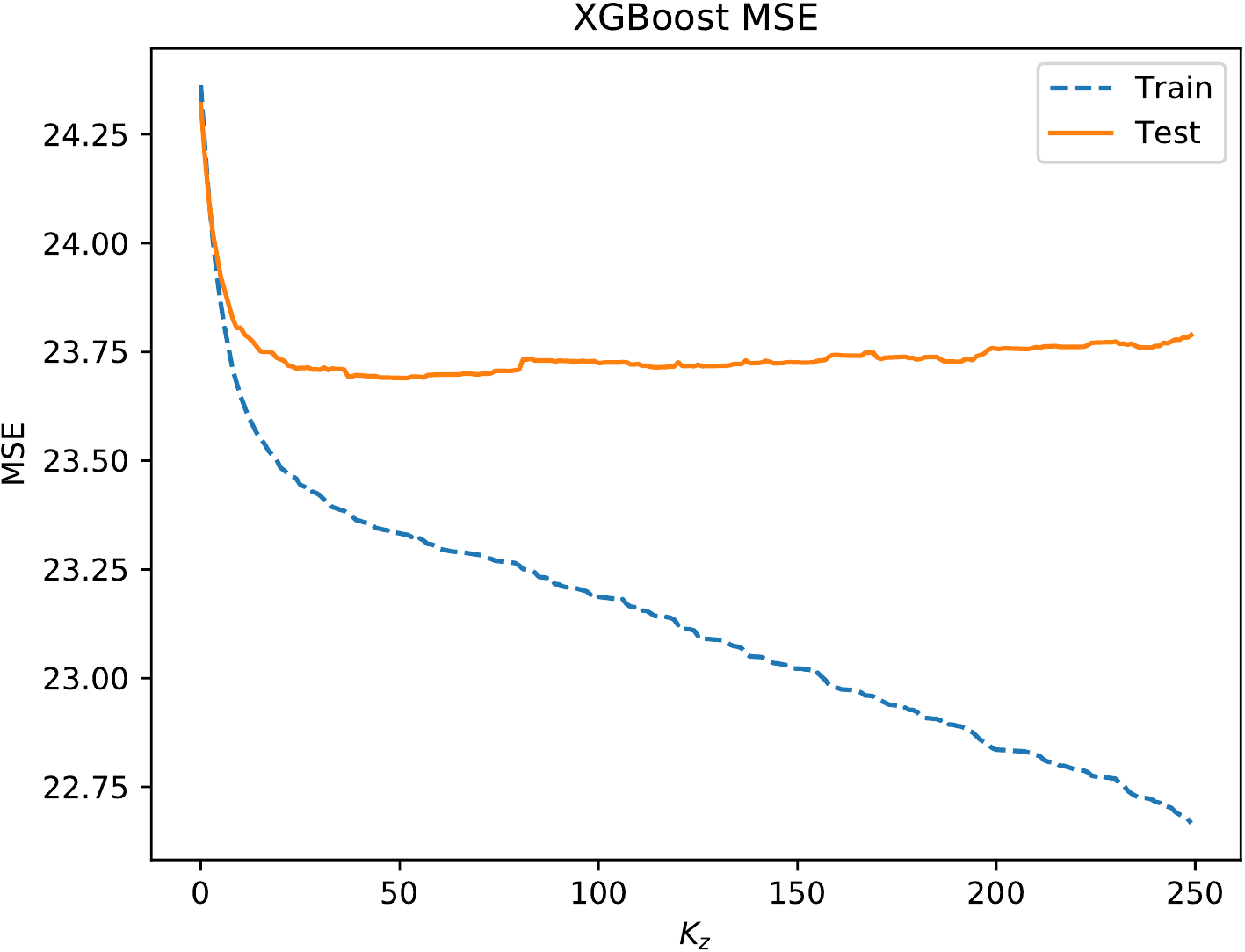}
		\subcaption{The XGBoost model for $Z.$}
	\end{subfigure}
	~ 
	\begin{subfigure}[b]{0.47\textwidth}
		\includegraphics[width=\textwidth]{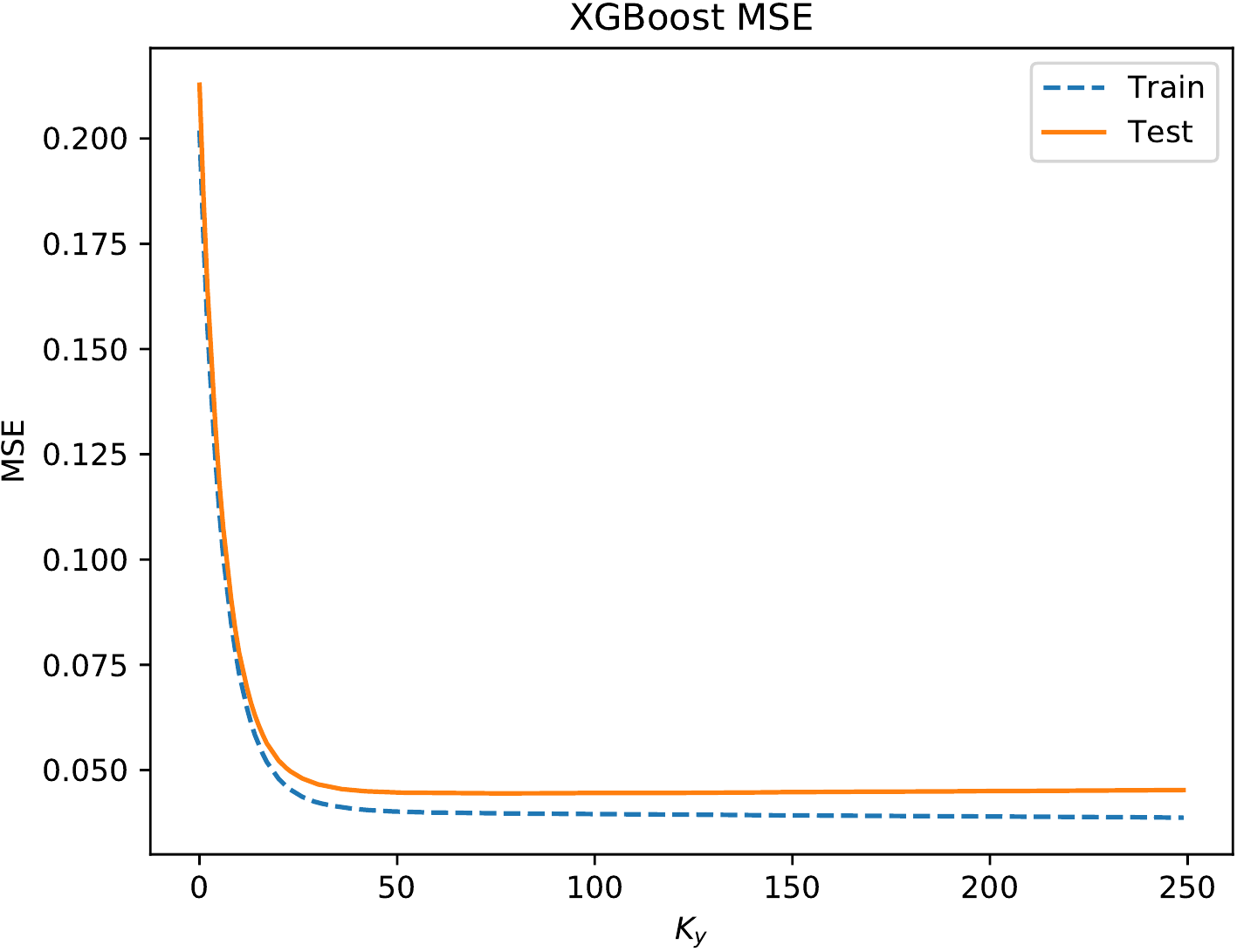}
		\subcaption{The XGBoost model for $Y.$}
	\end{subfigure}
	\caption{The MSEs of the XGBoost models in Scheme 1 for Example 6 for different numbers of trees, $N_T=10, M=10000, d=1$ and the learning rate is $0.1.$ }\label{fig:03}
\end{figure}

As indicated in \cite{Chassagneux2021, Hure2020}, the deep learning algorithm \cite{Han2017} fails when $d\geq 3.$ Furthermore, the two backward deep learning schemes of \cite{Hure2020} and deep learning schemes with sparse grids of \cite{Chassagneux2021} fails when $d\geq 8.$ We refer to Table 5 in \cite{Chassagneux2021} for detailed comparisons. In Table \ref{table:07} we show that our scheme works well even for $d=50.$
\begin{table}
	\centering 
	\small
	\begin{tabular}{|c|c|c|c|c|c|c|{c}r}
		\hline
		\multirow{1}{*}{$T=1$}&  Theoretical  & Numerical & \multirow{4}{*}{$error_y$(Std. dev.)} & \multirow{4}{*}{$K_z=K_y$} & \multirow{4}{*}{avg. runtime}\\
		$M=20000$&solution & approximation &  &  & \\
		\cline{1-1}\multirow{1}{*}{$d$}& \multirow{2}{*}{$Y_0$} & \multirow{2}{*}{$Y_0^{\Delta t}$} &  &  & \\
		\multirow{1}{*}{$N_T$}& &  & &  & \\
		\hline 
		$8$& \multirow{2}{*}{$1.16032$} & \multirow{2}{*}{$1.16830$} & \multirow{2}{*}{$0.01047(0.00931)$} & \multirow{2}{*}{$12$} & \multirow{2}{*}{$5.47$}   \\
		\multirow{1}{*}{$20$}&  &  &  &  &    \\
		\hline
		$10$& \multirow{2}{*}{$-0.21489$} & \multirow{2}{*}{$-0.21517$} & \multirow{2}{*}{$0.02435(0.03030)$} & \multirow{2}{*}{$40$} & \multirow{2}{*}{$14.19$}   \\
		\multirow{1}{*}{$20$}&  &  &  &  &    \\
		\hline
			$20$& \multirow{2}{*}{$0.25904$} & \multirow{2}{*}{$0.2555$} & \multirow{2}{*}{$0.02838(0.03492)$} & \multirow{2}{*}{$16$} & \multirow{2}{*}{$32.55$}   \\
		\multirow{1}{*}{$30$}&  &  &  &  &    \\
		\hline
			$50$& \multirow{2}{*}{$-0.47055$} & \multirow{2}{*}{$-0.47437$} & \multirow{2}{*}{$0.00667(0.00778)$} & \multirow{2}{*}{$10$} & \multirow{2}{*}{$1805.75$}   \\
		\multirow{1}{*}{$400$}&  &  &  &  &    \\
		\hline
	\end{tabular}
	\caption{The numerical approximation of $Y_0$ using Scheme 1 for Example 6 when $T=1.$}\label{table:07}
\end{table}
Note that we need to set $N_T=400$ for $d=50$ to obtain a good approximation, and our scheme shall work well also for a higher dimension if $\Delta t$ is sufficiently small. However, a higher dimension ($d > 50$) is not considered here due to the long computational time. The results can be further improved with a larger value of simple size.    
\section{Conclusion}\label{sec:conclusion}
In this work, we have proposed the XGBoost regression-based algorithms for numerically solving high-dimensional nonlinear BSDEs. We show how to use the XGBoost regression to approximate the conditional expectations arising by discretizing the time-integrands using the general theta-discretization method.
The time complexity and error analysis have been provided as well.
We have performed several numerical experiments for different types of BSDEs including $10000$-dimensional nonlinear problem. Our numerical results are quite promising and indicate that the proposed algorithms are very attractive to solve high-dimensional nonlinear BSDEs.

\section*{Acknowledgment}
The author gratefully acknowledges in-depth discussions with Prof.~Dr.~Hanno Gottschalk from the University of Wuppertal.

The author would like to thank Lorenc Kapllani for his assistance.
\bibliography{mybibfile}
\bibliographystyle{apalike}
\end{document}